\documentclass{article}
\usepackage{amsmath}
\usepackage{amssymb}
\usepackage{amsthm}
\usepackage{enumerate}
\usepackage{wasysym}
\usepackage{dsfont}
\usepackage{comment}
\usepackage{tipa}
\usepackage{tikz}
\usepackage{tikz-cd}
\usepackage{extarrows}
\usepackage{mathrsfs}
\usepackage{quiver}

\newtheorem{defn}{Definition}[section]
\newtheorem{prop}[defn]{Proposition}
\newtheorem{theo}[defn]{Theorem}
\newtheorem{lem}[defn]{Lemma}
\newtheorem{cor}[defn]{Corollary}
\newtheorem{rmk}[defn]{Remark}

\newtheorem*{claim*}{Claim}
\newtheorem*{just*}{Justification}
\newtheorem*{lem*}{Lemma}

\newcommand{\T}{\mathbb{T}}
\newcommand{\la}{\langle}
\newcommand{\ra}{\rangle}

\newcommand{\Term}{\mathsf{Term}}

\newcommand{\x}{\mathsf{x}}
\newcommand{\Sort}{\mathsf{Sort}}
\newcommand{\Fun}{\mathsf{Fun}}

\newcommand{\E}{\mathcal{E}}

\newcommand{\Z}{\mathcal{Z}}
\newcommand{\Tmod}{\T\mathsf{mod}}
\newcommand{\cod}{\mathsf{cod}}
\newcommand{\dom}{\mathsf{dom}}
\newcommand{\id}{\mathsf{id}}

\newcommand{\Aut}{\mathsf{Aut}}
\newcommand{\J}{\mathcal{J}}

\newcommand{\Group}{\mathsf{Group}}
\newcommand{\op}{\mathsf{op}}

\newcommand{\Set}{\mathsf{Set}}

\newcommand{\C}{\mathbb{C}}

\newcommand{\D}{\mathbb{D}}

\newcommand{\Ob}{\mathsf{Ob}}

\newcommand{\Arr}{\mathsf{Arr}}

\newcommand{\fraka}{\mathfrak{a}}
\newcommand{\frakb}{\mathfrak{b}}

\newcommand{\card}{\mathsf{card}}
\newcommand{\Sigmasite}{\Sigma^{(\C, \J)}}
\newcommand{\Tsite}{\T^{(\C, \J)}}

\newcommand{\Sh}{\mathsf{Sh}}
\newcommand{\yo}{\mathsf{y}}
\newcommand{\as}{\mathsf{a}}
\newcommand{\K}{\mathcal{K}}

\newcommand{\Horn}{\mathsf{Horn}}
\newcommand{\Dom}{\mathsf{Dom}}
\newcommand{\inn}{\mathsf{inn}}

\begin{document}

\title{Covariant Isotropy of Grothendieck Toposes}
\author{Jason Parker}

\maketitle

\begin{abstract}
We provide an explicit characterization of the \emph{covariant isotropy group} of any Grothendieck topos, i.e. the group of \emph{(extended) inner automorphisms} of any sheaf over a small site. As a consequence, we obtain an explicit characterization of the \emph{centre} of a Grothendieck topos, i.e. the automorphism group of its identity functor. 
\end{abstract}

\section{Introduction}

For any category $\C$, one can define its \emph{covariant isotropy group (functor)}
\[ \Z_\C : \C \to \Group, \] which sends any object $C \in \C$ to the group of natural automorphisms of the projection functor $C/\C \to \C$ from the slice category under $C$, and is one way of `functorializing' the assignment $C \mapsto \Aut(C)$, which is not functorial in general (unless $\C$ is a groupoid). Concretely, if we define $\mathsf{Dom}(C) := \{ f \in \Arr(\C) : \mathsf{dom}(f) = C\}$, then an element $\pi \in \Z_\C(C)$ is a $\mathsf{Dom}(C)$-indexed family of automorphisms \[ \pi = \left(\pi_f : \mathsf{cod}(f) \xrightarrow{\sim} \mathsf{cod}(f)\right)_{f \in \mathsf{Dom}(C)} \] with the following naturality property: for any $f : C \to C'$ and $f' : C' \to C''$ in $\C$ we have
\[ \pi_{f' \circ f} \circ f' = f' \circ \pi_f, \] as in the following commutative square:
\[\begin{tikzcd}
	{C'} && {C'} \\
	\\
	{C''} && {C''}
	\arrow["{\pi_f}", from=1-1, to=1-3]
	\arrow["{f'}"', from=1-1, to=3-1]
	\arrow["{\pi_{f' \circ f}}"', from=3-1, to=3-3]
	\arrow["{f'}", from=1-3, to=3-3]
\end{tikzcd}\]
Roughly speaking, an element $\pi \in \Z_\C(C)$ provides an automorphism $\pi_{\id_C} : C \xrightarrow{\sim} C$ that can be coherently or functorially \emph{extended} along any morphism out of $C$. 

The central motivation for studying covariant isotropy is that it encodes a notion of \emph{inner automorphism} or \emph{conjugation} for a category, which can be understood as follows. For any group $G$ and $g \in G$, let $\inn_g : G \xrightarrow{\sim} G$ be the inner automorphism of $G$ with $\inn_g(x) := gxg^{-1}$ for any $x \in G$. It is easy to verify that any inner automorphism of $G$ induces an element $\pi \in \Z_\Group(G)$: given an inner automorphism $\inn_g : G \xrightarrow{\sim} G$, one simply defines $\pi_f := \inn_{f(g)} : H \xrightarrow{\sim} H$ for any group homomorphism $f : G \to H$. Far less obvious is the fact, proven by George Bergman in \cite[Theorem 1]{Bergman}, that for any $\pi \in \Z_\Group(G)$, there is a unique element $g \in G$ with $\pi_{\id_G} = \inn_g$. In other words, the inner automorphisms are \emph{precisely} those group automorphisms that can be coherently or functorially extended along morphisms out of their domains. This provides a purely \emph{categorical} formulation of the notion of inner automorphism, which makes sense in any category. Motivated by these considerations, we may then refer to the elements of $\Z_\C(C)$ as the \emph{extended inner automorphisms} of $C \in \Ob(\C)$, while an \emph{inner automorphism} of $C$ is an automorphism $f : C \xrightarrow{\sim} C$ for which there is some extended inner automorphism $\pi \in \Z_\C(C)$ with $\pi_{\id_C} = f$.    

 In \cite{MFPSpaper}, the author and his collaborators studied the covariant isotropy group of the category $\Tmod$ of models of a single-sorted algebraic theory $\T$, and showed that it encodes the usual notion of inner automorphism for $\T$, if it has one. It is shown therein that the covariant isotropy group of $M \in \Tmod$ may be described in terms of the elements of $M\la \x \ra$ (the $\T$-model obtained from $M$ by freely adjoining an indeterminate element $\x$) that are \emph{substitutionally invertible} and \emph{commute generically with} the operations of $\T$. In the author's recent PhD thesis \cite{thesis}, this analysis was extended to the category of models of any finitary \emph{quasi-equational} or \emph{essentially algebraic} theory in the sense of \cite{Horn}, which is a multi-sorted equational theory in which certain operations may only be \emph{partially} defined. Generalizing the results from \cite{MFPSpaper}, it is shown in \cite{thesis} that the covariant isotropy group of a model $M$ of a finitary quasi-equational theory $\T$ can be described in terms of sort-indexed families in $\prod_C M\la \x_C \ra$ that are substitutionally invertible and commute generically with the operations of $\T$, where $M\la \x_C \ra$ is the $\T$-model obtained from $M$ by freely adjoining a new element $\x_C$ of sort $C$. This analysis then yielded explicit characterizations of the covariant isotropy groups of small strict monoidal categories and presheaf toposes. 

Generalizing from the latter class of examples, in the present paper we characterize the covariant isotropy group of any \emph{Grothendieck topos}. Recall that a Grothendieck topos is a category equivalent to the category $\Sh(\C, \J)$ of sheaves on a small site $(\C, \J)$. It is well-known that any Grothendieck topos is \emph{locally presentable} (cf. e.g. \cite[Proposition 3.4.16]{Borceux}), and hence is equivalent to the category of models of an (infinitary) essentially algebraic or quasi-equational theory (cf. e.g. \cite[Theorem 3.36]{LPAC}). To characterize the covariant isotropy group of any such topos, we therefore wish to use the logical methods for characterizing covariant isotropy developed in \cite{thesis}, which in particular rely on the result that any finitary quasi-equational theory $\T$ has an initial model, as explicitly and constructively proven in \cite[Theorem 22]{Horn} by the use of \emph{term models}. It is well-known that any infinitary essentially algebraic theory also has an initial model, provided that the theory is suitably \emph{bounded} by a regular cardinal (cf. again \cite[Theorem 3.36]{LPAC}), assuming a bounded version of the axiom of choice. Knowing this, we therefore first show in Section \ref{logicalbackground} that any bounded infinitary quasi-equational theory has an initial \emph{term model} as in \cite[Theorem 22]{Horn} (assuming a bounded axiom of choice), and we then use this fact to show that the logical methods for characterizing covariant isotropy developed in \cite{thesis} can be extended from finitary to bounded infinitary quasi-equational theories. In Section \ref{toposes}, we then show that any Grothendieck topos $\E$ is the category of models for a particular bounded infinitary quasi-equational theory, and then use the logical characterization of covariant isotropy for such theories developed in Section \ref{logicalbackground} to characterize the covariant isotropy of $\E$. As a consequence, we then deduce an explicit characterization of the \emph{centre} of $\E$, i.e. the automorphism group of its identity functor, since the centre of $\E$ is isomorphic to the covariant isotropy group of its initial object.     

\section{Covariant Isotropy of Locally Presentable Categories}
\label{logicalbackground}

In this section, we show that the techniques developed in \cite{thesis} for computing the covariant isotropy group of the category of models of any finitary quasi-equational theory, equivalently any locally finitely presentable category, can be extended to the category of models of any bounded infinitary quasi-equational theory, equivalently any locally presentable category. So we first define bounded infinitary quasi-equational theories, discuss their set-theoretic semantics, and show that they admit initial term models, assuming a bounded axiom of choice. The material in this section borrows heavily from \cite{Horn}, extended to the bounded infinitary context. 

First, we define the notion of a bounded infinitary \emph{first-order signature}. Recall that a \emph{regular cardinal} is a cardinal that is not the sum of any smaller number of smaller cardinals. 

\begin{defn}
\label{signature}
{\em 
Let $\lambda$ be a regular cardinal. A $\lambda$\emph{-ary (first-order) signature} $\Sigma$ is a pair of sets $\Sigma = (\Sigma_{\mathsf{Sort}}, \Sigma_{\mathsf{Fun}})$ such that:
\begin{itemize}
\item $\Sigma_{\mathsf{Sort}}$ is the set of \emph{sorts} of $\Sigma$.

\item $\Sigma_{\mathsf{Fun}}$ is the set of \emph{function symbols} of $\Sigma$. Each element $f \in \Sigma_{\mathsf{Fun}}$ comes equipped with a pair $\left((A_i)_{i < \fraka}, A\right)$, where $\fraka < \lambda$ is a cardinal number and $A, A_i$ are sorts for all $i < \fraka$. We write this as $f : \prod_{i < \fraka} A_i \to A$. In case $\fraka = 0$, we write $f : A$.
\end{itemize}
\noindent A \emph{bounded infinitary} signature is a $\lambda$-ary signature for some regular cardinal $\lambda$. \qed
}
\end{defn} 

\noindent Next, we define the class of \emph{terms} over a given $\lambda$-ary signature. Note that the class of terms over a $\lambda$-ary signature will in fact constitute a \emph{set} rather than a proper class, since there is only a \emph{set} of function symbols and the arity of each function symbol is bounded by the regular cardinal $\lambda$. 

\begin{defn}
\label{terms}
{\em Let $\Sigma$ be a $\lambda$-ary signature for some regular cardinal $\lambda$. For every sort $A \in \Sigma_{\mathsf{Sort}}$, we assume that we have a set $V_A$ of variables of sort $A$ of cardinality $\lambda$. We now define the set $\mathsf{Term}(\Sigma)$ of \emph{terms} of $\Sigma$ recursively as follows, while simultaneously defining the \emph{sort} and the set $\mathsf{FV}(t)$ of \emph{free variables} of a term $t \in \mathsf{Term}(\Sigma)$:
\begin{itemize}
\item If $A \in \Sigma_{\mathsf{Sort}}$ and $x \in V_A$, then $x \in \mathsf{Term}(\Sigma)$ is of sort $A$ with $\mathsf{FV}(x) := \{x\}$.

\item If $f : \prod_{i < \fraka} A_i \to A$ is a function symbol of $\Sigma$ and $t_i \in \mathsf{Term}(\Sigma)$ with $t_i : A_i$ for each $i < \fraka$, then $f\left[(t_i)_{i < \fraka}\right] \in \mathsf{Term}(\Sigma)$ is of sort $A$, and $\mathsf{FV}\left(f\left[(t_i)_{i < \fraka}\right]\right) := \bigcup_{i < \fraka} \mathsf{FV}(t_i)$. In particular, if $c$ is a constant symbol of sort $A$, then $c$ is a term of sort $A$ with $\mathsf{FV}(c) = \emptyset$.
\end{itemize}

\noindent If $t \in \mathsf{Term}(\Sigma)$ and $\mathsf{FV}(t) = \emptyset$, then we call $t$ a \emph{closed} term. If $t \in \mathsf{Term}(\Sigma)$, then we write $t(\vec{x})$ to mean that $\mathsf{FV}(t) \subseteq \vec{x}$. 
\qed
}
\end{defn}

\begin{defn}
\label{Hornformulas}
{\em Let $\Sigma$ be a $\lambda$-ary signature for some regular cardinal $\lambda$. We define the class $\mathsf{Horn}(\Sigma)$ of \emph{Horn formulas} over $\Sigma$ recursively as follows, while simultaneously defining the set $\mathsf{FV}(\varphi)$ of \emph{free variables} of a formula $\varphi \in \mathsf{Horn}(\Sigma)$:
\begin{itemize}
\item If $t_1, t_2 \in \mathsf{Term}(\Sigma)$ are terms of the same sort, then $t_1 = t_2 \in \mathsf{Horn}(\Sigma)$, and $\mathsf{FV}(t_1 = t_2) := \mathsf{FV}(t_1) \cup \mathsf{FV}(t_2)$.

\item $\top \in \mathsf{Horn}(\Sigma)$ (the empty conjunction), and $\mathsf{FV}(\top) := \emptyset$.

\item If $I$ is a set of cardinality $< \lambda$ and $\varphi_i \in \mathsf{Horn}(\Sigma)$ for each $i \in I$, then $\bigwedge_{i \in I} \varphi_i \in \mathsf{Horn}(\Sigma)$, and $\mathsf{FV}\left(\bigwedge_{i \in I} \varphi_i\right) := \bigcup_{i \in I} \mathsf{FV}(\varphi_i)$.
\end{itemize}
If $\varphi \in \mathsf{Horn}(\Sigma)$ and $\mathsf{FV}(\varphi) = \emptyset$, then we will refer to $\varphi$ as a (Horn) \emph{sentence}. If $\varphi \in \mathsf{Horn}(\Sigma)$, then we will write $\varphi(\vec{x})$ to mean that $\mathsf{FV}(\varphi) \subseteq \vec{x}$. \qed
}
\end{defn}

\begin{defn}
\label{sequent}
{\em Let $\Sigma$ be a $\lambda$-ary signature for some regular cardinal $\lambda$. A $\lambda$\emph{-ary Horn sequent} over $\Sigma$ is an expression of the form $\varphi \vdash^{\vec{x}} \psi$, where $\varphi, \psi \in \mathsf{Horn}(\Sigma)$ and $\mathsf{FV}(\varphi), \mathsf{FV}(\psi) \subseteq \vec{x}$ and $\vec{x}$ contains fewer than $\lambda$ variables. \qed
}
\end{defn}

\begin{defn}
{\em 
Let $\lambda$ be a regular cardinal. A $\lambda$-\emph{ary quasi-equational theory} $\T$ is a set of $\lambda$-ary Horn sequents over a $\lambda$-ary signature $\Sigma$. A \emph{bounded infinitary quasi-equational theory} is a $\lambda$-ary quasi-equational theory for some regular cardinal $\lambda$.} \qed
\end{defn}

One can now set up a deduction system of \emph{partial Horn logic} for bounded infinitary quasi-equational theories, wherein certain Horn sequents are designated as logical axioms, and there are logical inference rules allowing one to deduce certain Horn sequents from other Horn sequents. We refer the reader to \cite{Horn} for a list of all the specific logical axioms and inference rules of partial Horn logic, which must of course be generalized from the finitary context to the bounded infinitary context. The main distinguishing feature of this deduction system is that equality of terms is \emph{not} assumed to be reflexive, i.e. if $t(\vec{x})$ is a term over a given signature, then $\top \vdash^{\vec{x}} t(\vec{x}) = t(\vec{x})$ is \emph{not} a logical axiom of partial Horn logic, unless $t$ is a variable. In other words, if we abbreviate the equation $t = t$ by $t \downarrow$ (read: $t$ \emph{is defined}), then unless $t$ is a variable, the sequent $\top \vdash^{\vec{x}} t \downarrow$ is \emph{not} a logical axiom of partial Horn logic. 

If $\T$ is a $\lambda$-ary quasi-equational theory over a $\lambda$-ary signature $\Sigma$ and $\varphi \vdash^{\vec{x}} \psi$ is a $\lambda$-ary Horn sequent over $\Sigma$, then we say that the sequent $\varphi \vdash^{\vec{x}} \psi$ is \emph{provable} in $\T$ if there is a sequence of $\lambda$-ary Horn sequents (of length $< \lambda$) whose last member is $\varphi \vdash^{\vec{x}} \psi$, and each member of the sequence is either a logical axiom of partial Horn logic, an axiom of $\T$, or is obtained from previous elements of the sequence by an inference rule of partial Horn logic. We also say that $\T$ \emph{proves} the sequent $\varphi \vdash^{\vec{x}} \psi$, or that this sequent is a \emph{theorem} of $\T$. If $\T$ proves a Horn sequent of the form $\top \vdash^{\vec{x}} \varphi$, then we usually write this as $\T \vdash^{\vec{x}} \varphi$.    

We now review the set-theoretic semantics of partial Horn logic. We recall that if $A$ and $B$ are any sets, then a \emph{partial function} $f : A \rightharpoondown B$ is a total function $f : \mathsf{dom}(f) \to B$ with $\mathsf{dom}(f) \subseteq A$.  

\begin{defn}
{\em Let $\Sigma$ be a $\lambda$-ary signature for some regular cardinal $\lambda$. A (set-based) \emph{partial} $\Sigma$\emph{-structure} $M$ is given by the following data:

\begin{enumerate}

\item For every $A \in \Sigma_\Sort$, a set $M_A$. 

\item For every function symbol $f : \prod_{i < \fraka} A_i \to A$ of $\Sigma$, a \emph{partial function} $f^M : \prod_{i < \fraka} M_{A_i} \rightharpoondown M_{A}$. \qed
\end{enumerate}
}
\end{defn}

\begin{defn}
{\em Let $\Sigma$ be a $\lambda$-ary signature for some regular cardinal $\lambda$, and let $M$ and $N$ be partial $\Sigma$-structures. A $\Sigma$\emph{-morphism} $h : M \to N$ is a $\Sigma_{\mathsf{Sort}}$-indexed sequence of \emph{total} functions $h = (h_A : M_A \to N_A)_{A \in \Sigma_\Sort}$ satisfying the following condition:
\begin{itemize}

\item For any function symbol $f : \prod_{i < \fraka} A_i \to A$ in $\Sigma$ and any $(m_i)_{i < \fraka} \in \prod_{i < \fraka} M_{A_i}$, if $(m_i)_{i < \fraka} \in \mathsf{dom}\left(f^M\right)$, then $\left(h_{A_i}(m_i)\right)_{i < \fraka} \in \mathsf{dom}\left(f^N\right)$ and \[ h_A\left(f^M\left[(m_i)_{i < \fraka}\right]\right) = f^N\left(\left[h_{A_i}(m_i)\right]_{i < \fraka}\right) \in N_A. \]
\end{itemize} \qed }
\end{defn}

It is easy to verify that the (componentwise) composition of $\Sigma$-morphisms is a $\Sigma$-morphism, and that the sequence of identity functions $(\mathsf{id}_A : M_A \to M_A)_A$ is a $\Sigma$-morphism $\mathsf{id} : M \to M$ that is an identity for composition. So we can form the category $\mathsf{P}\Sigma\mathsf{Str}$ of partial $\Sigma$-structures and $\Sigma$-morphisms. 

Before we can define the notion of a (set-based) model of a bounded infinitary quasi-equational theory, we must first define the interpretations of terms and Horn formulas in partial structures.

\begin{defn}
{\em Let $\Sigma$ be a $\lambda$-ary signature for some regular cardinal $\lambda$. Let $t\left[(x_i)_{i < \fraka}\right] : A$ be an element of $\mathsf{Term}(\Sigma)$ with free variables among $x_i : A_i$ for $i < \fraka$, where $\fraka < \lambda$ is a cardinal. Let $M$ be a partial $\Sigma$-structure. We define the \emph{partial} function \[ t\left[(x_i)_{i < \fraka}\right]^M : \prod_{i < \fraka} M_{A_i} \rightharpoondown M_A \] by induction on the structure of $t$: 
\begin{itemize}
\item If $t \equiv x_i : A_i$ for some $i < \fraka$, then we set $t^M := \pi_i : \prod_{i < \fraka} M_{A_i} \to M_{A_i}$, the (total) projection onto the $i^{\mathsf{th}}$ factor.

\item If $t \equiv f\left[(t_j)_{j < \mathfrak{b}}\right] : B$ for some function symbol $f : \prod_{j < \frakb} B_j \to B$ of $\Sigma$ with $t_j \in \mathsf{Term}(\Sigma)$ and $t_j : B_j$ for each $j < \frakb$, we first set
\begin{align*}
\mathsf{dom}\left(t^M\right)	&:= \left\{ \vec{a} \in \bigcap_{j < \frakb} \mathsf{dom}\left(t_j^M\right) : \left(t_j^M(\vec{a})\right)_{j < \frakb} \in \mathsf{dom}\left(f^M\right)\right\},
\end{align*} 
and for any $\vec{a} \in \mathsf{dom}\left(t^M\right)$ we set \[ t^M(\vec{a}) := f^M\left(\left[t_j^M(\vec{a})\right]_{j < \frakb}\right) \in M_B, \] which defines $t^M = f\left[(t_j)_{j < \mathfrak{b}}\right]^M : \prod_{i < \fraka} M_{A_i} \rightharpoondown M_B$. \qed 
\end{itemize}
}
\end{defn}

\begin{defn}
{\em Let $\Sigma$ be a $\lambda$-ary signature for some regular cardinal $\lambda$, and let $\varphi\left[(x_i)_{i < \fraka}\right]$ be a Horn formula over $\Sigma$ with free variables among $x_i : A_i$ for $i < \fraka$ (where $\fraka < \lambda$ is a cardinal). Let $M$ be a partial $\Sigma$-structure. We define \[ \varphi\left[(x_i)_{i < \fraka}\right]^M \subseteq \prod_{i < \fraka} M_{A_i} \] by induction on the structure of $\varphi$: 
\begin{itemize}
\item If $\varphi \equiv t_1 = t_2$ for some terms $t_1, t_2$ of the same sort, then \[ \varphi^M = (t_1 = t_2)^M := \left\{ \vec{a} \in \mathsf{dom}\left(t_1^M\right) \cap \mathsf{dom}\left(t_2^M\right) : t_1^M(\vec{a}) = t_2^M(\vec{a})\right\}. \]

\item If $\varphi \equiv \top$, then $\top^M := \prod_{i < \fraka} M_{A_i}$.

\item If $\varphi \equiv \bigwedge_{k \in K} \varphi_k$ for some set $K$ with cardinality $< \lambda$ and Horn formulas $\varphi_k \in \Horn(\Sigma)$ for each $k \in K$, then \[ \left(\bigwedge_{k \in K} \varphi_k\right)^M := \bigcap_{k \in K} \varphi_k^M \subseteq \prod_{i < \fraka} M_{A_i}. \]
\end{itemize} \qed }
\end{defn}

\begin{defn}
{\em Let $\Sigma$ be a $\lambda$-ary signature for some regular cardinal $\lambda$, let $M$ be a partial $\Sigma$-structure, and let $\varphi(\vec{x}), \psi(\vec{x})$ be Horn formulas over $\Sigma$. Then $M$ \emph{models} or \emph{satisfies} the Horn sequent $\varphi \vdash^{\vec{x}} \psi$ if $\varphi(\vec{x})^M \subseteq \psi(\vec{x})^M$. \qed }
\end{defn}

\begin{defn}
{\em Let $\T$ be a $\lambda$-ary quasi-equational theory over a $\lambda$-ary signature $\Sigma$ for some regular cardinal $\lambda$, and let $M$ be a partial $\Sigma$-structure. Then $M$ is a \emph{model} of $\T$ if $M$ satisfies every axiom of $\T$. \qed }
\end{defn}

\noindent For a bounded infinitary quasi-equational theory $\T$ over a signature $\Sigma$, we now let $\Tmod$ be the full subcategory of $\mathsf{P}\Sigma\mathsf{Str}$ on the models of $\T$. 

In order to sketch the details of the Initial Model Theorem for bounded infinitary quasi-equational theories (cf. \cite[Theorem 22]{Horn}), we first require the following definitions.

\begin{defn}
{\em Let $\Sigma$ be a $\lambda$-ary signature for some regular cardinal $\lambda$ and $M$ a partial $\Sigma$-structure. For every sort $A$, let $\sim_A$ be an equivalence relation on $M_A$. Then the $\Sigma_{\mathsf{Sort}}$-indexed family of equivalence relations $(\sim_A)_A$ is a \emph{partial congruence} on $M$ if the following condition is satisfied:
\begin{itemize}
\item For every function symbol $f : \prod_{i < \fraka} A_i \to A$ in $\Sigma$ and all $\vec{a}, \vec{b} \in \prod_{i < \fraka} M_{A_i}$, if 
$a_i \sim_{A_i} b_i$ for all $i < \fraka$, then $\vec{a} \in \mathsf{dom}\left(f^M\right)$ iff $\vec{b} \in \mathsf{dom}\left(f^M\right)$ and \[ \vec{a}, \vec{b} \in \mathsf{dom}\left(f^M\right) \Longrightarrow f^M(\vec{a}) \sim_A f^M\left(\vec{b}\right). \] \qed
\end{itemize}
}
\end{defn}

\noindent We now have the following definition, which requires a bounded version of the axiom of choice for $\lambda$-small families of non-empty sets, which we call the $\lambda$\emph{-ary axiom of choice}.   

\begin{defn}
\label{partialcongruencestructure}
{\em Let $\Sigma$ be a $\lambda$-ary signature for some regular cardinal $\lambda$, and let $M$ be a partial $\Sigma$-structure. Let $\sim \ = (\sim_A)_A$ be a partial congruence on $M$. We define the \emph{partial quotient} $\Sigma$\emph{-structure} $M/{\sim}$ as follows: 

\begin{itemize}
\item For every sort $A \in \Sigma$, we set $(M/{\sim})_A := M_A/{\sim_A}$, the set of equivalence classes of $M_A$ modulo the equivalence relation $\sim_A$.

\item Assuming the $\lambda$-ary axiom of choice, for any function symbol $f : \prod_{i < \fraka} A_i \to A$ we set \[ \mathsf{dom}\left(f^{M/{\sim}}\right) := \left\{ \left([a_i]\right)_{i < \fraka} \in \prod_{i < \fraka} M_{A_i}/{\sim_{A_i}} : (a_i)_{i < \fraka} \in \mathsf{dom}\left(f^M\right) \right\}. \] Then for any $\left([a_i]\right)_{i < \fraka} \in \mathsf{dom}\left(f^{M/{\sim}}\right)$, we set \[ f^{M/{\sim}}\left(\left([a_i]\right)_{i < \fraka}\right) := \left[f^M\left[(a_i)_{i < \fraka}\right]\right]. \]

\end{itemize}
Because $\sim$ is a partial congruence on $M$, it easily follows that $M/{\sim}$ is a well-defined partial $\Sigma$-structure. \qed }
\end{defn}

\begin{rmk}
{\em
Let us comment briefly on the use of a bounded version of the axiom of choice in Definition \ref{partialcongruencestructure}. Under the circumstances of this definition, when defining the domain of $f^{M/{\sim}}$ for a function symbol $f : \prod_{i < \fraka} A_i \to A$ where $\fraka < \lambda$ is a cardinal, we need to \emph{choose} a representative from each congruence class of $M_{A_i}/{\sim_{A_i}}$ for each $i < \fraka$. While not fully constructive, this use of the restricted axiom of choice is nevertheless no more objectionable than the (bounded) use of the axiom of choice in the typical proof that if $F$ is a separated presheaf over a small site $(\C, \J)$, then the presheaf $F^+$ is a sheaf (cf. e.g. \cite[Lemma III.5.5]{MM}), which is a result that we will make use of below. However, we should also note that this use of a bounded axiom of choice is not essential, and can be replaced by a use of the axiom of collection, as shown in the proof of \cite[Theorem 4.2]{Berg}. Moreover, the typical construction of sheafification for Lawvere-Tierney topologies in elementary toposes (which subsumes sheafification for Grothendieck toposes) does not require the axiom of choice (cf. e.g. \cite[Theorem V.3.1]{MM}).

Of course, if the signature in question is \emph{finitary} (i.e. if $\lambda = \aleph_0$), then the axiom of choice is not needed in Definition \ref{partialcongruencestructure}. \qed    
} 
\end{rmk}

\noindent We now sketch the details of the Initial Model Theorem from \cite{Horn} that we will need for our purposes. From this point on in the paper, we will assume whatever bounded version of the axiom of choice that we need (depending on the regular cardinal at hand), without always explicitly mentioning this. 

First, given a bounded infinitary quasi-equational theory $\T$ over a signature $\Sigma$, we define a specific partial $\Sigma$-structure $M^{\T}$.
  
\begin{defn}
{\em Let $\Sigma$ be a $\lambda$-ary signature for some regular cardinal $\lambda$. First, let \[ \mathsf{Term}^c(\Sigma) := \{ t \in \mathsf{Term}(\Sigma) : \mathsf{FV}(t) = \emptyset \} \] be the set of \emph{closed terms} of $\mathsf{Term}(\Sigma)$. For any $A \in \Sigma_{\mathsf{Sort}}$, let
\[ \mathsf{Term}^c(\Sigma)_A := \{ t \in \mathsf{Term}^c(\Sigma) : t \ \text{is of sort} \ A \} \] be the set of closed $\Sigma$-terms of sort $A$. 

Now let $\T$ be a $\lambda$-ary quasi-equational theory over $\Sigma$. We define a partial $\Sigma$-structure $M^{\T}$ as follows:
\begin{itemize}
\item For any sort $A \in \Sigma$, we set 
\[ M^{\T}_A := \{ t \in \mathsf{Term}^c(\Sigma)_A : \T \vdash t \downarrow \}. \] 

\item For any function symbol $f : \prod_{i < \fraka} A_i \to A$ of $\Sigma$, we set \[ \mathsf{dom}\left(f^{M^{\T}}\right) := \left\{\vec{t} \in \prod_{i < \fraka} M^{\T}_{A_i} : \T \vdash f\left(\vec{t}\right)\downarrow\right\}, \] and if $\vec{t} \in \mathsf{dom}\left(f^{M^{\T}}\right)$, we set $f^{M^{\T}}\left(\vec{t}\right) := f\left(\vec{t}\right) \in M^\T_A$. \qed
\end{itemize}
}
\end{defn}

\noindent Now we define a partial congruence $\sim^\T$ on $M^{\T}$. For any sort $A \in \Sigma$, we set \[ \sim^\T_A  \: := \left\{ (t_1, t_2) \in M^{\T}_A \times M^{\T}_A : \T \vdash t_1 = t_2\right\}. \]

\noindent Using the rules of partial Horn logic, it is then straightforward to verify that $\sim^\T$ is in fact a partial congruence on $M^{\T}$. We now make the following definition:

\begin{defn}
{\em Let $\T$ be a $\lambda$-ary quasi-equational theory over a $\lambda$-ary signature $\Sigma$ for some regular cardinal $\lambda$, and let $M^{\T}$ be the partial $\Sigma$-structure and $\sim^\T$ the partial congruence on $M^{\T}$ just defined. Applying Definition \ref{partialcongruencestructure} and using the $\lambda$-ary axiom of choice, we then define the following partial $\Sigma$-structure: \[ \mathsf{Free}(\T) := M^{\T}/{\sim^\T}. \] \qed }
\end{defn}

\noindent The following theorem is then proven in \cite[Theorem 22]{Horn} for \emph{finitary} quasi-equational theories; its proof also applies to bounded infinitary quasi-equational theories, assuming the $\lambda$-ary axiom of choice.  

\begin{theo}
\label{initialmodelthm}
Let $\T$ be a bounded infinitary quasi-equational theory over a signature $\Sigma$. Then the partial $\Sigma$-structure $\mathsf{Free}(\T)$ is an \textbf{initial model} of $\T$, i.e. an initial object of the category $\Tmod$. \qed
\end{theo}

\begin{rmk}
\label{explicitdescriptionoffreemodel}
{\em
For concreteness, we give the explicit description of $\mathsf{Free}(\T)$ for a $\lambda$-ary quasi-equational theory $\T$ over a signature $\Sigma$. 
\begin{itemize}
\item For any sort $A \in \Sigma$, we have \[ \mathsf{Free}(\T)_A := M^{\T}_A/{\sim^\T_A} = \left\{ [t] : t \in \mathsf{Term}^c(\Sigma)_A \text{ and } \T \vdash t\downarrow \right\}, \]
where $[t]$ is the $\sim^\T_A$-congruence class of $t \in M^\T_A$ (so for any $s, t \in M^\T_A$, we have $[s] = [t]$ iff $\T \vdash s = t$).  

\item If $f : \prod_{i < \fraka} A_i \to A$ is a function symbol of $\Sigma$, then \[ \mathsf{dom}\left(f^{\mathsf{Free}(\T)}\right) = \left\{ \left([t_i]\right)_{i < \fraka} \in \prod_{i < \fraka} \mathsf{Free}(\T)_{A_i} : \T \vdash f\left[(t_i)_{i < \fraka}\right]\downarrow \right\}, \]
and for any $\left([t_i]\right)_{i < \fraka} \in \mathsf{dom}\left(f^{\mathsf{Free}(\T)}\right)$, we have \[ f^{\mathsf{Free}(\T)}\left(\left([t_i]\right)_{i < \fraka}\right) = \left[f\left[(t_i)_{i < \fraka}\right]\right]. \] \qed
\end{itemize}
}
\end{rmk}

It now follows that the main results of \cite[Section 2.2]{thesis}, characterizing the covariant isotropy group functor of $\Tmod$ for a finitary quasi-equational theory $\T$ in logical terms, extend to the bounded infinitary context (assuming a bounded axiom of choice). To restate these main results in the latter context, we first recall the following notions. 

If $M \in \Tmod$ for a $\lambda$-ary quasi-equational theory $\T$ over a signature $\Sigma$, then $\Sigma(M)$ is the \emph{diagram signature} of $M$, which extends $\Sigma$ by adding a new constant symbol $c_a : C$ for any sort $C \in \Sigma_\Sort$ and $a \in M_C$. The $\lambda$-ary quasi-equational theory $\T(M)$ over the signature $\Sigma(M)$ then extends $\T$ by adding axioms expressing that each new constant $c_a$ is defined, and that the function symbols of $\T$ interact with these constants appropriately (for explicit details, see \cite[Definition 2.2.3]{thesis}).  

If $\fraka < \lambda$ and $A_i \in \Sigma_\Sort$ for each $i < \fraka$, then $\Sigma(M, (\x_i)_{i < \fraka})$ is the signature that extends $\Sigma(M)$ by adding new pairwise distinct constant symbols $\x_i : A_i$ for all $i < \fraka$. The $\lambda$-ary quasi-equational theory $\T(M, (\x_i)_{i < \fraka})$ over the signature $\Sigma(M, (\x_i)_{i < \fraka})$ then extends $\T(M)$ by adding axioms expressing that $\x_i$ is defined for each $i < \fraka$.

Finally, if $C \in \Sigma_\Sort$, then $M\la \x_C \ra$ is defined to be the ($\Sigma$-reduct of) the initial model of $\T(M, \x_C)$, which therefore has the following explicit description (cf. Remark \ref{explicitdescriptionoffreemodel}): 

\begin{itemize}
\item For any $B \in \Sigma_\Sort$, \[ M\la \mathsf{x}_C \ra_B = \left\{ [t] : t \in \mathsf{Term}^c(\Sigma(M, \mathsf{x}_C))_B \wedge \T(M, \mathsf{x}_C) \vdash t\downarrow \right\}. \] 

\item If $f : \prod_{i < \fraka} A_i \to A$ is a function symbol of $\Sigma$, then \[ \mathsf{dom}\left(f^{M\la \x_C\ra}\right) = \left\{ \left([t_i]\right)_{i < \fraka} \in \prod_{i < \fraka} M\la \x_C \ra_{A_i} : \T(M, \x_C) \vdash f\left[(t_i)_{i < \fraka}\right]\downarrow \right\}, \]
and for any $\left([t_i]\right)_{i < \fraka} \in \mathsf{dom}\left(f^{M\la \x_C\ra}\right)$ we have \[ f^{M\la \x_C \ra}\left(\left([t_i]\right)_{i < \fraka}\right) = \left[f\left[(t_i)_{i < \fraka}\right]\right]. \]
\end{itemize}
   
\noindent We now have the following bounded infinitary version of \cite[Definition 2.2.47]{thesis}. The notion of syntactic substitution used in the following definition is the standard/expected one; cf. \cite[Remark 2.2.21]{thesis} for an explicit definition.

\begin{defn}
\label{commutesgenericallydefn}
{\em Let $\T$ be a $\lambda$-ary quasi-equational theory over a signature $\Sigma$, and let $M \in \Tmod$ and $([s_C])_C \in \prod_{C \in \Sigma_\Sort} M\la \mathsf{x}_C \ra_C$. 

\begin{itemize}

\item If $f : \prod_{i < \fraka} A_i \to A$ is a function symbol of $\Sigma$, then $([s_C])_C$ \emph{commutes generically with} $f$ if the Horn sequent \[ f\left[(\x_i)_{i < \fraka}\right] \downarrow \: \vdash \ s_A[f\left[(\x_i)_{i < \fraka}\right]/\mathsf{x}_A] = f\left[\left(s_{A_i}[\mathsf{x}_i/\mathsf{x}_{A_i}]\right)_{i < \fraka}\right] \] is provable in $\T\left(M, (\x_i)_{i < \fraka}\right)$. 

\item We say that $([s_C])_C$ is \emph{(substitutionally) invertible} if for every $B \in \Sigma_\Sort$ there is some $\left[s_B^{-1}\right] \in M\la \mathsf{x}_B \ra_B$ with \[ \left[s_B\left[s_B^{-1}/\mathsf{x}_B\right]\right] = [\mathsf{x}_B] = \left[s_B^{-1}\left[s_B/\mathsf{x}_B\right]\right] \in M\la \mathsf{x}_B \ra_B, \] i.e. with
\[ \T(M, \mathsf{x}_B) \vdash s_B\left[s_B^{-1}/\mathsf{x}_B\right] = \mathsf{x}_B = s_B^{-1}[s_B/\mathsf{x}_B]. \]

\item We say that $([s_C])_C$ \emph{reflects definedness} if for every function symbol $f : \prod_{i < \fraka} A_i \to A$ in $\Sigma$, the sequent  
\[ f\left[\left(s_{A_i}[\mathsf{x}_i/\mathsf{x}_{A_i}]\right)_{i < \fraka}\right] \downarrow \ \ \vdash f\left[(\x_i)_{i < \fraka}\right] \downarrow \] is provable in $\T\left(M, (\x_i)_{i < \fraka}\right)$. \qed
\end{itemize}
}
\end{defn}

\noindent As in \cite[Definition 2.2.36]{thesis}, we then have a functor $G_\T : \Tmod \to \Group$, with $G_\T(M)$ for $M \in \Tmod$ being the group of all elements $([s_C])_C \in \prod_{C \in \Sigma} M\la \mathsf{x}_C \ra_C$ that are substitutionally invertible and commute generically with and reflect definedness of every function symbol of $\Sigma$. The unit of this group is the element $([\x_C])_C$, the inverse of any element is obtained via the substitutional invertibility of each of its components (as in Definition \ref{commutesgenericallydefn}), and if $([s_C])_C, ([t_C])_C \in G_\T(M)$, then their product is obtained via substitution as
\[ ([s_C])_C \cdot ([t_C])_C := \left(\left[s_C\left[t_C/\x_C\right]\right]\right)_C. \] For more details, see \cite[Propositions 2.2.35, 2.2.38]{thesis}. From the bounded infinitary versions of \cite[Theorems 2.2.41, 2.2.53]{thesis}, we then conclude that if $\T$ is a bounded infinitary quasi-equational theory, then \[ \Z_{\Tmod} \cong G_\T : \Tmod \to \Group. \] In other words, the covariant isotropy group of $M \in \Tmod$, i.e. its group of extended inner automorphisms, is isomorphic to the group of all elements of $([s_C])_C \in \prod_{C \in \Sigma} M\la \mathsf{x}_C \ra_C$ that are substitutionally invertible and commute generically with and reflect definedness of all operations of $\T$ (naturally in $M$). 

\section{Grothendieck Toposes}
\label{toposes}

In this section, we will use the logical characterization of covariant isotropy for bounded infinitary quasi-equational theories given in Section \ref{logicalbackground} to characterize the covariant isotropy group functor of any Grothendieck topos. To do this, we first associate to any small site $(\C, \J)$, i.e. a small category $\C$ equipped with a Grothendieck topology $J$ (defined in terms of covering sieves), a bounded infinitary signature and quasi-equational theory whose models will be the sheaves for the site (inspired by \cite[3.1]{Bridge}). We assume that the reader has familiarity with the basic notions of the theory of Grothendieck toposes. 

\begin{defn}
\label{sheafsignature}
{\em
Let $(\C, \J)$ be a small site, and let $\lambda$ be the smallest regular cardinal greater than $\card(J)$ for any $C \in \Ob(\C)$ and $J \in \J(C)$. We define a $\lambda$-ary signature $\Sigma^{(\C, \J)}$ as follows:
\begin{itemize}
\item We set $\Sigma^{(\C, \J)}_\Sort := \Ob(\C)$.

\item We define $\Sigma^{(\C, \J)}_\Fun$ to have the following function symbols:
\begin{itemize}
\item For any morphism $f : C \to D$ in $\C$, there is a function symbol $\alpha_f : D \to C$ in $\Sigmasite_\Fun$.

\item For any $C \in \Ob(\C)$ and $J \in \J(C)$, there is a function symbol
\[ \sigma_J : \prod_{f \in J} \dom(f) \to C. \] \qed  
\end{itemize}
\end{itemize}
} 
\end{defn} 

\begin{defn}
\label{sheaftheory}
{\em
Let $(\C, \J)$ be a small site, and $\lambda$ the smallest regular cardinal greater than $\card(J)$ for every $C \in \Ob(\C)$ and $J \in \J(C)$. We define a $\lambda$-ary quasi-equational theory $\Tsite$ over the signature $\Sigmasite$ to have the following axioms:
\begin{itemize}
\item For any morphism $f : C \to D$ in $\C$, the axiom $\top \vdash^{x : D} \alpha_f(x) \downarrow$.

\item For any $C \in \Ob(\C)$, the axiom $\top \vdash^{x : C} \alpha_{\id_C}(x) = x$. 

\item For any composable morphisms $C \xrightarrow{f} D \xrightarrow{g} E$ in $\C$, the axiom \linebreak $\top \vdash^{x : E} \alpha_{g \circ f}(x) = \alpha_f\left(\alpha_g(x)\right)$.

\item For any $C \in \Ob(\C)$ and $J \in \J(C)$, the axiom
\[ \bigwedge_{f \in J} \left(\bigwedge_{\cod(g) = \dom(f)} \alpha_g\left(x_f\right) = x_{f \circ g}\right) \vdash^{\vec{x}} \sigma_J\left[(x_f)_{f \in J}\right] \downarrow \wedge \bigwedge_{f \in J} \alpha_f\left(\sigma_J\left[(x_f)_{f \in J}\right]\right) = x_f, \]
where $\vec{x}$ is the context consisting of pairwise distinct variables $x_f : \dom(f)$ for all $f \in J$.

\item For any $C \in \Ob(\C)$ and $J \in \J(C)$, the axiom
\[ \bigwedge_{f \in J} \left(\bigwedge_{\cod(g) = \dom(f)} \alpha_g\left(x_f\right) = x_{f \circ g}\right) \ \wedge \ \left(\bigwedge_{f \in J} \alpha_f(y) = x_f\right) \vdash^{\vec{x}, y} \sigma_J\left[(x_f)_{f \in J}\right] = y, \]
where $\vec{x}$ is the same context as above, and $y : C$. \qed  
\end{itemize} 
}
\end{defn} 

\noindent It is then straightforward to verify that we have an isomorphism of categories $\Tsite\mathsf{mod} \cong \Sh(\C, \J)$, where $\Sh(\C, \J)$ is the Grothendieck topos of sheaves on $(\C, \J)$. The first three groups of axioms ensure that any model $F$ of $\Tsite$ is a functor $F : \C^\op \to \Set$, the penultimate group ensures that any matching family of elements of $F$ has an amalgamation, and the last group ensures that any such amalgamation is unique, so that $F$ is indeed a sheaf for the site $(\C, \J)$.

We are now going to characterize the covariant isotropy group $\Z_{\Sh(\C, \J)}: \Sh(\C, \J) \to \Group$ by first characterizing the naturally isomorphic $G_{\Tsite} : \Tsite\mathsf{mod} \cong \Sh(\C, \J) \to \Group$ (cf. the end of Section \ref{logicalbackground}). We first restrict our attention to small sites $(\C, \J)$ that are \emph{subcanonical}, meaning that every representable presheaf $\C(-, C) : \C^\op \to \Set$ for $C \in \Ob(\C)$ is a sheaf, and satisfy the further property that no object is covered by the empty sieve (i.e. no $C \in \Ob(\C)$ satisfies $\varnothing \in \J(C)$). We first require the following technical lemmas, whose proofs are mostly deferred to the Appendix for the reader's convenience.

For the following lemma, recall that a presheaf $F : \C^\op \to \Set$ is \emph{separated} (with respect to the topology $\J$) if any matching family in $F$ has at most one amalgamation.  

\begin{lem}
\label{separatedlemma}
Let $(\C, \J)$ be a small site in which no object is covered by the empty sieve, and let $F, G \in \Sh(\C, \J)$. Then the coproduct presheaf $F + G$ is separated. 
\end{lem}

\begin{proof}
Assume the hypotheses, and suppose without loss of generality that $F, G$ are (pointwise) disjoint, so that $(F + G)(C) = F(C) \bigcup G(C)$ for any $C \in \Ob(\C)$. Let $C \in \Ob(\C)$ and $J \in \J(C)$, and let us show that any matching family for $J$ in $F + G$ has at most one amalgamation. So let $(x_f)_{f \in J}$ be such a matching family, so that $x_f \in F(\dom(f)) \bigcup G(\dom(f))$ for any $f \in J$. We then have three possibilities:
\begin{itemize}
\item Suppose that $x_f \in F(\dom(f))$ for every $f \in J$. Then $(x_f)_{f \in J}$ is a matching family for $J$ in $F$, and since $F$ is a sheaf, this matching family has a unique amalgamation $x \in F(C)$, so that $x$ is also an amalgamation for this matching family in $F + G$. If $y \in F(C) \bigcup G(C)$ were another amalgamation, then we could not have $y \in G(C)$, because then, choosing some $f \in J$ (which is possible, since $J \neq \varnothing$), we would have $G(f)(y) = x_f$, which is impossible, since $G(f)(y) \in G(\dom(f))$ and $x_f \in F(\dom(f))$ and these sets are disjoint. So we must have $y \in F(C)$, which entails that $y$ is another amalgamation for $(x_f)_{f \in J}$ in $F$, and hence $x = y$ because $F$ is a sheaf. So in this case $(x_f)_{f \in J}$ has a unique amalgamation in $F + G$, and hence at most one amalgamation. 

\item If $x_f \in G(\dom(f))$ for every $f \in J$, then reasoning as in the previous case shows that this matching family again has a unique amalgamation in $F + G$.

\item In the third case, we have some distinct $g, h \in J$ with $x_g \in F(\dom(g))$ and $x_h \in G(\dom(h))$. The disjointness of $F$ and $G$ then entails that $(x_f)_{f \in J}$ has \emph{no} amalgamation in $F + G$, and hence has at most one.     
\end{itemize} 
\end{proof}

\noindent The following lemma essentially says that if the diagram theory of a sheaf (over a site satisfying our assumptions) proves that  the terms representing two parallel morphisms are equal, then the morphisms themselves must be equal. Recall from the discussion preceding Definition \ref{commutesgenericallydefn} that $\Tsite(F, \x_C)$ is the diagram theory of $F \in \Sh(\C, \J)$ extended by a new (provably defined) constant $\x_C : C$.  

\begin{lem}
\label{arrowequality}
Let $(\C, \J)$ be a small subcanonical site in which no object is covered by the empty sieve, and let $F \in \Sh(\C, \J)$. If $f, g : D \to C$ are parallel morphisms in $\C$ with $\Tsite(F, \x_C) \vdash \alpha_f(\x_C) = \alpha_g(\x_C)$, then $f = g$.   
\end{lem} 

\begin{proof}
See the Appendix.     
\end{proof} 

\noindent Given $F \in \Sh(\C, \J)$ and $C \in \Ob(\C)$, we say that a closed term $t \in \Term^c\left(\Sigmasite(F, \x_C)\right)$ is \emph{pure} if it contains no object constant $c_a$ from $F$ (recall from the discussion preceding Definition \ref{commutesgenericallydefn} that $\Sigmasite(F, \x_C)$ is the diagram signature of $F$ extended by a new constant $\x_C : C$). We can now prove the following essential result.  

\begin{lem}
\label{purenormalformlemma}
Let $(\C, \J)$ be a small subcanonical site in which no object is covered by the empty sieve, and let $F \in \Sh(\C, \J)$ and $C \in \Ob(\C)$. For any pure closed term $t \in \Term^c\left(\Sigmasite(F, \x_C)\right)$ with $\Tsite(F, \x_C) \vdash t \downarrow$ and $t : D$ for some $D \in \Ob(\C)$, there is some morphism $f : D \to C$ in $\C$ with \[ \Tsite(F, \x_C) \vdash t = \alpha_f(\x_C). \] 
\end{lem}

\begin{proof}
See the Appendix.
\end{proof}

\noindent The following lemma states that no term representing a morphism can be provably equal to an object constant. 

\begin{lem}
\label{constantlemma}
Let $(\C, \J)$ be a small site in which no object is covered by the empty sieve, and let $F \in \Sh(\C, \J)$ and $C \in \Ob(\C)$. For any morphism $f \in \Arr(\C)$ with $\cod(f) = C$, there is no $a \in F(\dom(f))$ with $\Tsite(F, \x_C) \vdash \alpha_f(\x_C) = c_a$. 
\end{lem}

\begin{proof}
See the Appendix.
\end{proof}

\noindent The following lemma essentially provides a kind of \emph{normal form} for provably defined closed terms over a sheaf. 

\begin{lem}
\label{normalformlemma}
Let $(\C, \J)$ be a small site, and let $F \in \Sh(\C, \J)$ and $C \in \Ob(\C)$. For any closed term $t \in \Term^c\left(\Sigmasite(F, \x_C)\right)$ with $\Tsite(F, \x_C) \vdash t \downarrow$ and $t : D$ for some $D \in \Ob(\C)$, there is some cover $J \in \J(D)$ with $\Tsite(F, \x_C) \vdash t = \sigma_J\left(\left(t_h\right)_{h \in J}\right)$, where for any $h \in J$ either $\Tsite(F, \x_C) \vdash t_h = c_a$ for some $a \in F(\dom(h))$ or $\Tsite(F, \x_C) \vdash t_h = \alpha_f(\x_C)$ for some morphism $f : \dom(h) \to C$ in $\C$. 
\end{lem}

\begin{proof}
See the Appendix.
\end{proof}

\noindent This final technical lemma essentially states that any provably defined closed term with a substitutional right inverse is provably equal to a pure term (i.e. a term containing no object constant).

\begin{lem}
\label{invertiblelemma}
Let $(\C, \J)$ be a small subcanonical site in which no object is covered by the empty sieve, and let $F \in \Sh(\C, \J)$ and $C \in \Ob(\C)$. For any closed term $t \in \Term^c\left(\Sigmasite(F, \x_C)\right)$ with $\Tsite(F, \x_C) \vdash t \downarrow$ and $t : C$, if there is some term $s \in \Term^c\left(\Sigmasite(F, \x_C)\right)$ with $\Tsite(F, \x_C) \vdash s \downarrow$ and $s : C$ and $\Tsite(F, \x_C) \vdash t[s/\x_C] = \x_C$, then there is a \textbf{pure} term $t' \in \Term^c\left(\Sigmasite(F, \x_C)\right)$ with $\Tsite(F, \x_C) \vdash t = t'$.  
\end{lem}

\begin{proof}
See the Appendix.
\end{proof}

\noindent We can now characterize in \emph{logical} terms the covariant isotropy group of any sheaf over a small site that satisfies the assumptions of Lemma \ref{invertiblelemma}. For such a sheaf $F$, recall that the logical variant $G_{\Tsite}(F)$ of the covariant isotropy group $\Z_{\Sh(\C, \J)}(F)$ of $F$ consists of all $\Ob(\C)$-indexed families $([s_C])_C \in \prod_{C \in \Ob(\C)} F\la \x_C \ra_C$ that are substitutionally invertible and commute generically with and reflect definedness of all function symbols of $\Tsite$; cf. Definition \ref{commutesgenericallydefn} and the discussion thereafter. 

If $\C$ is a small category, then $\Aut\left(1_\C\right)$ is the group of natural automorphisms of the identity functor $1_\C : \C \to \C$, sometimes referred to as the \emph{centre} of $\C$.  

\begin{prop}
\label{isotropyprop}
Let $(\C, \J)$ be a small subcanonical site in which no object is covered by the empty sieve. For any $F \in \Sh(\C, \J)$, we have
\[ G_{\Tsite}(F) = \left\{ \left(\left[\alpha_{\psi_C}(
\x_C)\right]\right)_{C} \in \prod_{C \in \Ob(\C)} F\la \x_C \ra_C \colon \psi \in \Aut\left(1_\C\right) \right\}. \]
\end{prop}

\begin{proof}
Assume the hypotheses. If $\psi$ is a natural automorphism of $1_\C : \C \to \C$, then for any $C \in \Ob(\C)$ we have $\alpha_{\psi_C}(\x_C) \in \Term^c\left(\Sigmasite(F, \x_C)\right)$ with $\alpha_{\psi_C}(\x_C) : C$ and $\Tsite(F, \x_C) \vdash \alpha_{\psi_C}(\x_C) \downarrow$, so that $\left[\alpha_{\psi_C}(\x_C)\right] \in F\la \x_C \ra_C$ and hence we do indeed have $\left(\left[\alpha_{\psi_C}(
\x_C)\right]\right)_{C} \in \prod_{C \in \Ob(\C)} F\la \x_C \ra_C$. It is then straightforward to show that 
$\left(\left[\alpha_{\psi_C}(
\x_C)\right]\right)_{C}$ is substitutionally invertible and commutes generically with and reflects definedness of the operations of $\Tsite$, and therefore belongs to $G_{\Tsite}(F)$ (in fact, it suffices to just show that it is substitutionally invertible and commutes generically with the operations $\alpha_f$ for $f \in \Arr(\C)$, because a natural isomorphism between sheaves automatically preserves and reflects (amalgamations of) matching families). Categorically speaking, the right-to-left inclusion says that if $\psi : 1_\C \xrightarrow{\sim} 1_\C$ is a natural automorphism, then one can define an element $\pi \in \Z_{\Sh(\C, \J)}(F)$ by setting $\pi_\gamma : G \xrightarrow{\sim} G$ as $\left(\pi_\gamma\right)_C := G\left(\psi_C\right) : GC \xrightarrow{\sim} GC$ for any $C \in \Ob(\C)$ and natural transformation $\gamma : F \to G$ in $\Sh(\C, \J)$, and it is easy to check that this does indeed define an element of covariant isotropy for $F$. 

For the converse inclusion, let $\left(\left[s_C\right]\right)_C \in G_{\Tsite}(F)$. So for any $C \in \Ob(\C)$ we have $s_C \in \Term^c\left(\Sigmasite(F, \x_C)\right)$ with $s_C : C$ and $\Tsite(F, \x_C) \vdash s_C \downarrow$. Since each $s_C$ is substitutionally invertible, it follows by Lemma \ref{invertiblelemma} that each $s_C$ is (provably equal to) a \emph{pure} term. By Lemma \ref{purenormalformlemma}, it then follows that for every $C \in \Ob(\C)$, there is a morphism $\psi_C : C \to C$ in $\C$ with
\[ \Tsite(F, \x_C) \vdash s_C = \alpha_{\psi_C}(\x_C) \] (since $s_C : C$). We now claim that $\psi := \left(\psi_C\right)_{C \in \Ob(\C)}$ is a natural automorphism of $1_\C$, which will complete the proof. 

First, we show that each $\psi_C$ is an isomorphism. It follows as for $s_C$ that if $\left[s_C^{-1}\right]$ is the substitutional inverse of $[s_C]$, then there is some morphism $\psi_C^{-1} : C \to C$ in $\C$ with $\Tsite(F, \x_C) \vdash s_C^{-1} = \alpha_{\psi_C^{-1}}(\x_C)$, so that
\[ \Tsite(F, \x_C) \vdash \alpha_{\psi_C}\left(\alpha_{\psi_C^{-1}}(\x_C)\right) = \x_C = \alpha_{\psi_C^{-1}}\left(\alpha_{\psi_C}(\x_C)\right) \] and hence
\[ \Tsite(F, \x_C) \vdash \alpha_{\psi_C^{-1} \circ \psi_C}(\x_C) = \x_C = \alpha_{\id_C}(\x_C), \] so that $\psi_C^{-1} \circ \psi_C = \id_C$ by Lemma \ref{arrowequality}, and similarly $\psi_C \circ \psi_C^{-1} = \id_C$. This proves that $\psi_C$ is an isomorphism for every $C \in \Ob(\C)$.

To show that $\psi$ is natural, let $f : C \to D$ be a morphism in $\C$, and let us show that $f \circ \psi_C = \psi_D \circ f$. Since $\left(\left[s_C\right]\right)_C \in G_{\Tsite}(F)$, we know that $\left(\left[s_C\right]\right)_C$ commutes generically with the function symbol $\alpha_f : D \to C$ of $\Sigmasite$, which means that
\[ \Tsite(F, \x_D) \vdash s_C\left[\alpha_f(\x_D)/\x_C\right] = \alpha_f(s_D), \] i.e. that
\[ \Tsite(F, \x_D) \vdash \alpha_{\psi_C}(\x_C)\left[\alpha_f(\x_D)/\x_C\right] = \alpha_f\left(\alpha_{\psi_D}(\x_D)\right), \] so that
\[ \Tsite(F, \x_D) \vdash \alpha_{f \circ \psi_C}(\x_D) = \alpha_{\psi_D \circ f}(\x_D), \] from which we conclude $f \circ \psi_C = \psi_D \circ f$ by Lemma \ref{arrowequality}. This concludes the proof that $\psi$ is a natural automorphism of $1_\C$ with $([s_C])_C = \left(\left[\alpha_{\psi_C}(
\x_C)\right]\right)_{C}$.   
\end{proof}

\noindent We can now deduce a \emph{categorical} characterization of the covariant isotropy group of any sheaf over a small site satisfying the assumptions of Proposition \ref{isotropyprop}:

\begin{cor}
\label{isotropycor}
Let $(\C, \J)$ be a small subcanonical site in which no object is covered by the empty sieve. For any $F \in \Sh(\C, \J)$, we have
\[ \Z_{\Sh(\C, \J)}(F) \cong \Aut\left(1_\C\right). \]
\end{cor}

\begin{proof}
By the natural isomorphism $\Z_{\Sh(\C, \J)} \cong G_{\Tsite}$ (cf. the discussion following Definition \ref{commutesgenericallydefn}), it suffices to show that $G_{\Tsite}(F) \cong \Aut\left(1_\C\right)$. By Proposition \ref{isotropyprop}, we have a surjective assignment $G_{\Tsite}(F) \to \Aut\left(1_\C\right)$ given by $\left(\left[\alpha_{\psi_C}(
\x_C)\right]\right)_{C} \mapsto \psi$, which is well-defined by Lemma \ref{arrowequality}. It is clearly injective, and it is easy to see that it preserves group multiplication, so that it is the desired group isomorphism.    
\end{proof}

We have now shown that if $(\C, \J)$ is a small subcanonical site in which no object is covered by the empty sieve, then the covariant isotropy group functor of $\Sh(\C, \J)$ is constant on $\Aut\left(1_\C\right)$. We now wish to show that this result is still true when the assumption that no object is covered by the empty sieve is removed, and that a modified version of this result is true when the site is not assumed to be subcanonical. 

To prove the first of these claims, let $(\C, \J)$ be a small subcanonical site. If $C \in \Ob(\C)$ satisfies $\varnothing \in \J(C)$, it then follows that $C$ must be an initial object of $\C$. To see this, let $X \in \Ob(\C)$: then the representable presheaf $\C(-, X)$ is a sheaf because $(\C, \J)$ is subcanonical. Since $\varnothing \in \J(C)$, it then follows that $\C(C, X)$ has exactly one element (the amalgamation of the empty matching family), so that there is a unique morphism from $C$ to $X$, and $C$ is indeed initial in $\C$. We now have:

\begin{lem}
\label{emptycoverlemma}
Let $(\C, \J)$ be a small subcanonical site, and let $\D$ be the full subcategory of $\C$ generated by the objects $D \in \Ob(\C)$ with $\varnothing \notin \J(D)$. Then $\Aut\left(1_\C\right) \cong \Aut\left(1_\D\right)$. 
\end{lem} 

\begin{proof}
Since $\D$ is a full subcategory of $\C$, there is a canonical group homomorphism $\varphi : \Aut\left(1_\C\right) \to \Aut\left(1_\D\right)$ given by restriction to $\D$, so it remains to show that $\varphi$ is bijective. To show injectivity, let $\psi \in \Aut\left(1_\C\right)$ with $\psi_D = \id_D$ for every $D \in \Ob(\D)$, and let us show that $\psi = 1_\C$, i.e. that $\psi_C = \id_C$ for every $C \in \Ob(\C)$ with $\varnothing \in \J(C)$. But if $\varnothing \in \J(C)$, then we have just shown that $C$ must be initial, so that we must indeed have $\psi_C = \id_C$, as desired. To show surjectivity, let $\chi \in \Aut\left(1_\D\right)$. We then define $\psi \in \Aut\left(1_\C\right)$ by setting $\psi_D := \chi_D$ for any $D \in \Ob(\D)$ and $\psi_C := \id_C$ for any $C \in \Ob(\C) \setminus \Ob(\D)$, as we must, since $\varnothing \in \J(C)$ and hence $C$ is initial. Then $\psi$ is clearly natural and $\varphi(\psi) = \chi$.          
\end{proof}

\noindent We can now show that the conclusion of Corollary \ref{isotropycor} remains true if the small subcanonical site $(\C, \J)$ has objects covered by the empty sieve:

\begin{prop}
\label{emptycoverprop}
Let $(\C, \J)$ be a small subcanonical site. Then the covariant isotropy group functor $\Z_{\Sh(\C, \J)} : \Sh(\C, \J) \to \Group$ is constant on $\Aut\left(1_\C\right)$.
\end{prop}

\begin{proof}
Let $\D$ be the full subcategory of $\C$ generated by the objects $D \in \Ob(\C)$ with $\varnothing \notin \J(D)$, and let $\K$ be the Grothendieck topology on $\D$ induced by the Grothendieck topology $\J$ on $\C$, so that a sieve $S$ on $D \in \Ob(\D)$ belongs to $\K(D)$ iff the sieve on $D$ that $S$ generates in $\C$ belongs to $\J(D)$ (cf. \cite[Page 112]{MM}). Since $\J$ is subcanonical, it follows by the proof of \cite[Appendix, Corollary 4.3]{MM} that $\K$ is subcanonical. We now claim that every object $C$ of $\C$ has a $\J$-cover by objects of $\D$, i.e. that there is some $J \in \J(C)$ such that every morphism of $J$ factors through some object of $\D$. This is certainly true if $C \in \Ob(\D)$, considering (e.g.) the maximal sieve on $C$. And if $C \notin \Ob(\D)$, so that $\varnothing \in \J(C)$, then $C$ trivially has a $\J$-cover by objects of $\D$, namely the empty sieve. So by the Comparison Lemma (cf. \cite[Appendix, Corollary 4.3]{MM}), we obtain a canonical equivalence of categories $\Sh(\C, \J) \simeq \Sh(\D, \K)$. It now suffices to show that $\Z_{\Sh(\D, \K)} : \Sh(\D, \K) \to \Group$ is constant on $\Aut\left(1_\C\right)$, since covariant isotropy is invariant under equivalence (cf. \cite[8.4]{Funk}). To achieve this, it suffices by Corollary \ref{isotropycor} and Lemma \ref{emptycoverlemma} to show that no object of $\D$ is $\K$-covered by the empty sieve. But if $D \in \Ob(\D)$ and $\varnothing \in \K(D)$, then by definition of $\K$ this would imply that $\varnothing \in \J(D)$, which contradicts the definition of $\D$.     
\end{proof}

From Proposition \ref{emptycoverprop} we can now deduce an explicit characterization (in terms of (extended) inner automorphisms) of the covariant isotropy group of any sheaf over a small subcanonical site. Recall that an \emph{extended inner automorphism} of $F \in \Sh(\C, \J)$ is just an element of $\Z_{\Sh(\C, \J)}(F)$, i.e. a natural automorphism of the projection functor $F/\Sh(\C, \J) \to \Sh(\C, \J)$, while an \emph{inner automorphism} of $F$ is an automorphism $\gamma : F \xrightarrow{\sim} F$ with $\gamma = \pi_{\id_F}$ for some extended inner automorphism $\pi$ of $F$.

\begin{cor}
\label{subcanonicalcor}
Let $(\C, \J)$ be a small subcanonical site, and let $F \in \Sh(\C, \J)$. 
\begin{itemize}
\item If $\pi = \left(\pi_\gamma : \cod(\gamma) \to \cod(\gamma)\right)_{\gamma \in \Dom(F)}$ is a (not necessarily natural) $\Dom(F)$-indexed family of endomorphisms in $\Sh(\C, \J)$, then $\pi \in \Z_{\Sh(\C, \J)}(F)$ iff there is a unique $\psi \in \Aut\left(1_{\C}\right)$ with 
\[ \left(\pi_\gamma\right)_C = G(\psi_C) : GC \xrightarrow{\sim} GC \] for any natural transformation $\gamma : F \to G$ in $\Sh(\C, \J)$. 

\item A natural endomorphism $\gamma : F \to F$ is an inner automorphism iff there is some $\psi \in \Aut\left(1_\C\right)$ with $\gamma_C = F(\psi_C)$ for every $C \in \Ob(\C)$.  
\end{itemize}
\end{cor}

\begin{proof}
The first assertion follows from Proposition \ref{emptycoverprop} and \cite[Corollary 2.2.42]{thesis}. For the second assertion, if $\gamma$ is an inner automorphism, then there is some $\pi \in \Z_{\Sh(\C, \J)}(F)$ with $\gamma = \pi_{\id_F}$. From the first assertion, we deduce the existence of some $\psi \in \Aut\left(1_\C\right)$ with $\gamma_C = \left(\pi_{\id_F}\right)_C = F(\psi_C)$ for every $C \in \Ob(\C)$, as desired. Conversely, suppose that there is some $\psi \in \Aut\left(1_\C\right)$ with $\gamma_C = F(\psi_C)$ for every $C \in \Ob(\C)$. For any natural transformation $\theta : F \to G$ in $\Sh(\C, \J)$, define $\pi_\theta : G \xrightarrow{\sim} G$ by $\left(\pi_\theta\right)_C := G(\psi_C)$ for every $C \in \Ob(\C)$. Then from the first assertion we have $\pi := \left(\pi_\theta\right)_{\theta \in \Dom(F)} \in \Z_{\Sh(\C, \J)}(F)$ with $\pi_{\id_F} = \gamma$, so that $\gamma$ is an inner automorphism.   
\end{proof}

\noindent Having removed the assumption that no object is covered by the empty sieve, we now remove the assumption that the site is subcanonical (while obtaining a modified result). First, we require the following lemma. Recall that a full subcategory $\C$ of a category $\E$ is \emph{dense} if for any $E \in \Ob(\E)$, the canonical cocone $(f)_{f \in \C/E}$ for the projection functor $\Gamma_E : \C/E \to \E$ is a colimit cocone, where $\C/E$ is the full subcategory of the slice category $\E/E$ on those morphisms with domain in $\C$. 

\begin{lem}
\label{denselemma}
Let $\E$ be a (locally small) category with small full dense subcategory $\C \hookrightarrow \E$. Then $\Aut\left(1_\C\right) \cong \Aut\left(1_\E\right)$. 
\end{lem} 

\begin{proof}
We have a canonical group homomorphism $\varphi : \Aut\left(1_\E\right) \xrightarrow{\sim} \Aut\left(1_\C\right)$ given by restriction to the full subcategory $\C \hookrightarrow \E$. To show that $\varphi$ is injective, let $\alpha \in \Aut\left(1_\E\right)$ with $\alpha_C = \id_C$ for every $C \in \Ob(\C)$, and let us show that $\alpha_E = \id_E$ for every $E \in \Ob(\E)$. Since the set $\Ob(\C)$ is in particular a generator of $\E$, it suffices to show that $\alpha_E \circ f = \id_E \circ f = f$ for every morphism $f : C \to E$ with $C \in \Ob(\C)$. But by hypothesis and naturality of $\alpha$ we do indeed have $\alpha_E \circ f = f \circ \alpha_C = f \circ \id_C = f$. So $\varphi$ is injective. 

To show that $\varphi$ is surjective, let $\beta \in \Aut\left(1_\C\right)$, and let us extend $\beta$ to a natural automorphism $\alpha : 1_\E \xrightarrow{\sim} 1_\E$. Let $E \in \Ob(\E)$ be arbitrary, and let $\C/E$ be the full subcategory of $\E/E$ generated by the morphisms $f : C \to E$ with $C \in \Ob(\C)$. We have the projection functor $\Gamma_E : \C/E \to \E$ and canonical cocone $(f)_{f \in \C/E}$ on $\Gamma_E$ with vertex $E$, which is a colimit cocone because $\C$ is dense. We also have another cocone $\left(f \circ \beta_{\dom(f)} : \dom(f) \to E\right)_{f \in \C/E}$ on $\Gamma_E$ with vertex $E$, since for any commutative triangle
\[\begin{tikzcd}
	C && {C'} \\
	& E
	\arrow["g", from=1-1, to=1-3]
	\arrow["f"', from=1-1, to=2-2]
	\arrow["{f'}", from=1-3, to=2-2]
\end{tikzcd}\]
in $\E$ with $C, C' \in \Ob(\C)$ we have
\[ f' \circ \beta_{C'} \circ g = f' \circ g \circ \beta_C = f \circ \beta_C \] by naturality of $\beta$. So we obtain a unique factorization $\alpha_E : E \to E$ of this cocone through the colimit cocone $(f)_{f \in \C/E}$, with $\alpha_E \circ f = f \circ \beta_C$ for any $f : C \to E$ in $\C/E$.  

Now we show that $\alpha_E$ is an isomorphism. Analogously to how we obtained $\alpha_E$, we obtain a unique morphism $\alpha_E^{-1} : E \to E$ with $\alpha_E^{-1} \circ f = f \circ \beta_C^{-1}$ for any $f : C \to E$ in $\C/E$. Then for any $f : C \to E$ in $\C/E$ we have
\[ \alpha_E \circ \alpha_E^{-1} \circ f = \alpha_E \circ f \circ \beta_C^{-1} = f \circ \beta_C \circ \beta_C^{-1} = f \circ \id_C = f, \] so that $\alpha_E \circ \alpha_E^{-1} = \id_E$ because $\Ob(\C)$ is a generator, and similarly we have $\alpha_E^{-1} \circ \alpha_E = \id_E$, so that $\alpha_E$ is indeed an isomorphism. 

For any $C \in \Ob(\C)$ we do indeed have $\alpha_C = \beta_C$, since we have $\id_C : C \to C$ in $\C/C$ and hence $\alpha_C = \alpha_C \circ \id_C = \id_C \circ \beta_C = \beta_C$ by definition of $\alpha_C$. So $\alpha$ does extend $\beta$, and it remains to show that $\alpha$ is natural. So let $h : E \to E'$ be a morphism in $\E$, and let us show that $h \circ \alpha_E = \alpha_{E'} \circ h : E \to E'$. For any $f : C \to E$ in $\C/E$ we have
\[ h \circ \alpha_E \circ f = h \circ f \circ \beta_C = \alpha_{E'} \circ h \circ f \] by the definitions of $\alpha_E$ and $\alpha_{E'}$, so that the desired naturality equation holds because $\Ob(\C)$ is a generator of $\E$. This completes the proof that $\varphi : \Aut\left(1_\E\right) \xrightarrow{\sim} \Aut\left(1_\C\right)$ is surjective and hence a group isomorphism.       
\end{proof}

\noindent We can now remove the assumption that the site is subcanonical. If $(\C, \J)$ is a small site, then we write $\as\yo\C$ for the full subcategory of $\Sh(\C, \J)$ generated by the objects in the image of $\C \xrightarrow{\yo} \Set^{\C^\op} \xrightarrow{\as} \Sh(\C, \J)$, i.e. by the sheafifications of the representable presheaves.

\begin{theo}
\label{arbitraryprop}
Let $(\C, \J)$ be an arbitrary small site. 
\begin{itemize}
\item The covariant isotropy group functor $\Z_{\Sh(\C, \J)} : \Sh(\C, \J) \to \Group$ is constant on $\Aut\left(1_{\as\yo\C}\right)$.

\item The centre $\Aut\left(1_{\Sh(\C, \J)}\right)$ of $\Sh(\C, \J)$ is isomorphic to $\Aut\left(1_{\as\yo\C}\right)$. In particular, if $(\C, \J)$ is subcanonical, then $\Aut\left(1_{\Sh(\C, \J)}\right) \cong \Aut\left(1_\C\right)$.    
\end{itemize} 
\end{theo}

\begin{proof}
It is well-known (cf. \cite[Appendix, Corollary 4.2]{MM}) that there is a subcanonical topology $\K$ on the small full subcategory $\as\yo\C \hookrightarrow \Sh(\C, \J)$ for which $\Sh(\C, \J) \simeq \Sh(\as\yo\C, \K)$. Since covariant isotropy is invariant under equivalence (cf. \cite[8.4]{Funk}), we then obtain by Proposition \ref{emptycoverprop} that $\Z_{\Sh(\C, \J)}$ is constant on $\Aut\left(1_{\as\yo\C}\right)$. It is also well-known that the small full subcategory $\as\yo\C \hookrightarrow \Sh(\C, \J)$ is dense (cf. \cite[Page 139]{MM}), so that the first part of the second assertion follows from Lemma \ref{denselemma}. The second part of the second assertion then follows because $(\C, \J)$ being subcanonical implies that the composite $\C \xrightarrow{\yo} \Set^{\C^\op} \xrightarrow{\as} \Sh(\C, \J)$ is full and faithful, so that $\as\yo\C \cong \C$. 
\end{proof}

\begin{rmk}
\label{presheafresult}
{\em
In \cite[Corollary 5.3.6]{thesis}, the author showed that the covariant isotropy group functor of any \emph{presheaf} topos $\Set^{\C^\op}$ is constant on $\Aut\left(1_\C\right)$. Since $\Set^{\C^\op} = \Sh(\C, T)$ where $T$ is the trivial (subcanonical) Grothendieck topology on $\C$ (where the only covering sieves are the maximal ones), the aforementioned result for presheaf toposes is now a special case of Proposition \ref{emptycoverprop} and Theorem \ref{arbitraryprop}. \qed
}
\end{rmk} 

\begin{rmk}
\label{contravariant}
{\em 
Theorem \ref{arbitraryprop} illustrates a major difference between the covariant isotropy group $\Sh(\C, \J) \to \Group$ and the \emph{contravariant} isotropy group $\Sh(\C, \J)^\op \to \Group$ (which sends any $F \in \Sh(\C, \J)$ to the automorphism group of the projection functor $\Sh(\C, \J)/F \to \Sh(\C, \J)$). As shown in \cite[4.3]{Funk}, the latter functor is generally \emph{not} constant, and is in fact \emph{representable} by a sheaf of groups $Z : \C^\op \to \Group$. \qed   
}
\end{rmk} 

\begin{rmk}
\label{openquestion}
{\em
If the site $(\C, \J)$ is \emph{not} subcanonical, then in general there is no relation between $\Aut\left(1_\C\right)$ and $\Aut\left(1_{\as\yo\C}\right)$. For (an extreme) example, let $\C$ be any small category whose centre $\Aut\left(1_\C\right)$ is non-trivial (e.g. $\C$ could be the one-object category corresponding to any non-trivial abelian group), and let $\J$ be the topology on $\C$ in which \emph{every} sieve (including the empty sieve) is covering. Then the only sheaf for this site is the terminal presheaf, so that $\J$ is not subcanonical (because otherwise $\C$ would be (equivalent to) a preorder, contradicting the fact that $\Aut\left(1_\C\right)$ is non-trivial), and moreover $\Sh(\C, \J)$ is the trivial category. So then $\Aut\left(1_{\as\yo\C}\right) \cong \Aut\left(1_{\Sh(\C, \J)}\right)$ is trivial, even though $\Aut\left(1_\C\right)$ was assumed non-trivial. 
However, it is an open question whether it is somehow possible to characterize $\Aut\left(1_{\as\yo\C}\right)$ in terms of both $\Aut\left(1_\C\right)$ \emph{and} the topology $\J$. \qed  
} 
\end{rmk}

\begin{rmk}
\label{syntacticapproach}
{\em 
The reader may wonder why we have chosen a mostly logical or syntactic approach to characterizing the covariant isotropy group of a Grothendieck topos. While it may have been possible to develop a more purely \emph{categorical} proof, since the logical techniques for computing covariant isotropy in \cite{MFPSpaper} and \cite{thesis} were already available (albeit requiring extension to the bounded infinitary context), we felt that these would yield the most expedient proof. Perhaps more relevant, however, is the fact that in future work we hope to characterize the covariant isotropy group of the category of sheaves valued in $\Tmod$ for a (single-sorted) quasi-equational theory $\T$ (e.g. sheaves of groups or rings), and more generally, the covariant isotropy group of the category of models of one quasi-equational theory valued in the category of models of another such theory. To achieve this, it seems that a logical or syntactic approach will be optimal, if not essential. Hence, with these latter aims in mind, we decided to pursue a mostly logical approach in this paper. \qed
} 
\end{rmk}

\section{Appendix}

\begin{lem*}[\ref{arrowequality}]
Let $(\C, \J)$ be a small subcanonical site in which no object is covered by the empty sieve, and let $F \in \Sh(\C, \J)$. If $f, g : D \to C$ are parallel morphisms in $\C$ with $\Tsite(F, \x_C) \vdash \alpha_f(\x_C) = \alpha_g(\x_C)$, then $f = g$.   
\end{lem*} 

\begin{proof}
Assume the hypotheses, let $\yo : \C \to \Set^{\C^\op}$ be the Yoneda embedding, and let $\as : \Set^{\C^\op} \to \Sh(\C, \J)$ be the associated sheaf functor, defined in terms of the plus-construction as $\as(G) := G^{++}$ for any presheaf $G$ (cf. e.g. \cite[Section III.5]{MM}). Since $(\C, \J)$ is subcanonical, we know that $\yo C \in \Sh(\C, \J)$, from which it follows by Lemma \ref{separatedlemma} that the coproduct presheaf $F + \yo C$ is separated. By \cite[Lemma III.5.5]{MM}, it then follows that $(F + \yo C)^+ \in \Sh(\C, \J)$. Recall that for any $D \in \Ob(\C)$, the set $(F + \yo C)^+(D)$ is the set of matching families in the presheaf $F + \yo C$ for covers in $\J(D)$ modulo the equivalence relation which identifies two such matching families $(x_f)_{f \in R}$ and $(y_g)_{g \in S}$ when there is a cover $T \in \J(D)$ with $T \subseteq R \cap S$ and $x_h = y_h$ for all $h \in T$. We then have a canonical natural transformation $\eta : F \to (F + \yo C)^+$ defined as follows: for any $D \in \Ob(\C)$, the function $\eta_D : F(D) \to (F + \yo C)^+(D)$ sends $d \in F(D)$ to the equivalence class of the matching family $(F(f)(d))_{f \in t_D}$, where $t_D \in \J(D)$ is the maximal sieve. 

Since $\Tsite(F, \x_C) \vdash \alpha_f(\x_C) = \alpha_g(\x_C)$ by assumption and $\eta : F \to (F + \yo C)^+$ is a natural transformation and hence a morphism in $\Sh(\C, \J)$, it follows by (the bounded infinitary version of) \cite[Lemma 3.1.2]{thesis} that \[ (F + \yo C)^+(f) = (F + \yo C)^+(g) : (F + \yo C)^+(C) \to (F + \yo C)^+(D). \] The equivalence class $\left[(h)_{h \in t_C}\right]$ of the matching family $(h)_{h \in t_C}$ in $\yo C$ belongs to $(F + \yo C)^+(C)$, and hence we have
\[ (F + \yo C)^+(f)\left(\left[(h)_{h \in t_C}\right]\right) = (F + \yo C)^+(g)\left(\left[(h)_{h \in t_C}\right]\right), \] i.e.
\[ \left[(fh)_{h \in t_D}\right] = \left[(gh)_{h \in t_D}\right] \in (F + \yo C)^+(D). \] So there is a cover $T \in \J(D)$ with $f \circ h = g \circ h$ for every $h \in T$. We also have a matching family $(fh)_{h \in T}$ for $T \in \J(D)$ in $\yo C$. Since $\yo C$ is a sheaf and in particular separated, there is at most one morphism $k : D \to C$ in $\C$ with $k \circ h = f \circ h$ for every $h \in T$. But $f$ and $g$ both satisfy this property, so that we must have $f = g$, as desired.      
\end{proof} 

\begin{lem*}[\ref{purenormalformlemma}]
Let $(\C, \J)$ be a small subcanonical site in which no object is covered by the empty sieve, and let $F \in \Sh(\C, \J)$ and $C \in \Ob(\C)$. For any pure closed term $t \in \Term^c\left(\Sigmasite(F, \x_C)\right)$ with $\Tsite(F, \x_C) \vdash t \downarrow$ and $t : D$ for some $D \in \Ob(\C)$, there is some morphism $f : D \to C$ in $\C$ with \[ \Tsite(F, \x_C) \vdash t = \alpha_f(\x_C). \] 
\end{lem*}

\begin{proof}
Assume the hypotheses, and let us prove the claim by induction on the structure of $t \in \Term^c\left(\Sigmasite(F, \x_C)\right)$ with $t$ pure and $\Tsite(F, \x_C) \vdash t \downarrow$.
\begin{itemize}
\item If $t \equiv \x_C : C$, then $\Tsite(F, \x_C) \vdash t \downarrow$ and we clearly have $\Tsite(F, \x_C) \vdash \x_C = \alpha_{\id_C}(\x_C)$.

\item Suppose $t \equiv \alpha_g(t') : D$ for some morphism $g : D \to D'$ in $\C$ and some pure $t' \in \Term^c\left(\Sigmasite(F, \x_C)\right)$ of sort $D' \in \Ob(\C)$, and suppose moreover that $\Tsite(F, \x_C) \vdash \alpha_g(t') \downarrow$, so that $\Tsite(F, \x_C) \vdash t' \downarrow$. Then by the induction hypothesis, there is some morphism $h : D' \to C$ in $\Ob(\C)$ with $\Tsite(F, \x_C) \vdash t' = \alpha_h(\x_C)$, and hence we obtain
\[ \Tsite(F, \x_C) \vdash t = \alpha_g(t') = \alpha_g(\alpha_h(\x_C)) = \alpha_{h \circ g}(\x_C), \] as desired.  

\item Suppose that $t \equiv \sigma_J\left(\left(t_h\right)_{h \in J}\right) : D$ for some $D \in \Ob(\C)$ and $J \in \J(D)$ and pure $t_h \in \Term^c\left(\Sigmasite(F, \x_C)\right)$ of sort $\dom(h)$ for every $h \in J$, and suppose also that $\Tsite(F, \x_C) \vdash \sigma_J\left(\left(t_h\right)_{h \in J}\right) \downarrow$, so that $\Tsite(F, \x_C) \vdash t_h \downarrow$ for every $h \in J$. Moreover, it follows that $\Tsite(F, \x_C) \vdash \alpha_g\left(t_h\right) = t_{h \circ g}$ for every $h \in J$ and $g \in \Arr(\C)$ with $\cod(g) = \dom(h)$, and that $\Tsite(F, \x_C) \vdash \alpha_h(t) = t_h$ for every $h \in J$.  

By the induction hypothesis, we know for every $h \in J$ that there is some morphism $f_h : \dom(h) \to C$ with $\Tsite(F, \x_C) \vdash t_h = \alpha_{f_h}(\x_C)$. So for any $h \in J$ and $g \in \Arr(\C)$ with $\cod(g) = \dom(h)$, we have \[ \Tsite(F, \x_C) \vdash \alpha_{f_h \circ g}(\x_C) = \alpha_g\left(\alpha_{f_h}(\x_C)\right) = \alpha_g(t_h) = t_{h \circ g} = \alpha_{f_{h \circ g}}(\x_C). \] By Lemma \ref{arrowequality}, it then follows that $f_h \circ g = f_{h \circ g}$ for every $h \in J$ and $g \in \Arr(\C)$ with $\cod(g) = \dom(h)$. This means that the family of morphisms $\left(f_h\right)_{h \in J}$ is matching in $\yo C = \C(-, C)$ for the cover $J \in \J(D)$. Since $\yo C$ is a sheaf by assumption, it follows that there is a unique morphism $f : D \to C$ in $\C$ with $f \circ h = f_h$ for every $h \in J$. We now show that $\Tsite(F, \x_C) \vdash \sigma_J\left(\left(t_h\right)_{h \in J}\right) = \alpha_f(\x_C)$, completing the proof. By the axiom of $\Tsite$ expressing the uniqueness of amalgamations of matching families, it suffices to show that
$\Tsite(F, \x_C) \vdash \alpha_h\left(\alpha_f(\x_C)\right) = t_h$ for every $h \in J$. But we have
\[ \Tsite(F, \x_C) \vdash \alpha_h\left(\alpha_f(\x_C)\right) = \alpha_{f \circ h}(\x_C) = \alpha_{f_h}(\x_C) = t_h, \] as desired.           
\end{itemize}  
\end{proof}

\begin{lem*}[\ref{constantlemma}]
Let $(\C, \J)$ be a small site in which no object is covered by the empty sieve, and let $F \in \Sh(\C, \J)$ and $C \in \Ob(\C)$. For any morphism $f \in \Arr(\C)$ with $\cod(f) = C$, there is no $a \in F(\dom(f))$ with $\Tsite(F, \x_C) \vdash \alpha_f(\x_C) = c_a$. 
\end{lem*}

\begin{proof}
Assume the hypotheses, let $f : D \to C$, and suppose towards a contradiction that there were some $a \in F(D)$ with $\Tsite(F, \x_C) \vdash \alpha_f(\x_C) = c_a$. Then for any $G \in \Sh(\C, \J)$ and natural transformation $\gamma : F \to G$, it follows by \cite[Lemma 3.1.2]{thesis} that the function $G(f) : G(C) \to G(D)$ is constant on $\gamma_{D}(a) \in G(D)$. Let $\mathds{1} : \C^\op \to \Set$ be the terminal sheaf constant on the singleton $\{\ast\}$, and consider the coproduct presheaf $F + \mathds{1}$ (assuming without loss of generality that $F$ and $\mathds{1}$ are pointwise disjoint). By Lemma \ref{separatedlemma}, it follows that $F + \mathds{1}$ is separated, so that $(F + \mathds{1})^+$ is a sheaf by \cite[Lemma III.5.5]{MM}. We have a natural transformation $\eta : F \to (F + \mathds{1})^+$ defined as follows: for any $X \in \Ob(\C)$, the function $\eta_X : F(X) \to (F + \mathds{1})^+(X)$ sends $x \in F(X)$ to the equivalence class of the matching family $(F(h)(x))_{h \in t_X}$, where $t_X \in \J(X)$ is the maximal sieve. To obtain a contradiction and complete the proof, it suffices to show that $(F + \mathds{1})^+(f) : (F + \mathds{1})^+(C) \to (F + \mathds{1})^+(D)$ is \emph{not} constant on $\eta_D(a) = \left[(F(h)(a))_{h \in t_D}\right] \in (F + \mathds{1})^+(D)$. We have $\left[(\ast)_{h \in t_C}\right] \in (F + \mathds{1})^+(C)$. If we had
\[ (F + \mathds{1})^+(f)\left(\left[(\ast)_{h \in t_C}\right]\right) = \left[(F(h)(a))_{h \in t_D}\right], \] so that $\left[(F(h)(a))_{h \in t_D}\right] = \left[(\ast)_{h \in t_D}\right]$, then there would be some cover $J \in \J(D)$ with $F(k)(a) = \ast$ for all $k \in J$, which is impossible, since $J \neq \varnothing$ and $F$ and $\mathds{1}$ are disjoint. So $(F + \mathds{1})^+(f)\left(\left[(\ast)_{h \in t_C}\right]\right) \neq \left[(F(h)(a))_{h \in t_D}\right]$, as desired.             
\end{proof}

\begin{lem*}[\ref{normalformlemma}]
Let $(\C, \J)$ be a small site, and let $F \in \Sh(\C, \J)$ and $C \in \Ob(\C)$. For any closed term $t \in \Term^c\left(\Sigmasite(F, \x_C)\right)$ with $\Tsite(F, \x_C) \vdash t \downarrow$ and $t : D$ for some $D \in \Ob(\C)$, there is some cover $J \in \J(D)$ with $\Tsite(F, \x_C) \vdash t = \sigma_J\left(\left(t_h\right)_{h \in J}\right)$, where for any $h \in J$ either $\Tsite(F, \x_C) \vdash t_h = c_a$ for some $a \in F(\dom(h))$ or $\Tsite(F, \x_C) \vdash t_h = \alpha_f(\x_C)$ for some morphism $f : \dom(h) \to C$ in $\C$. 
\end{lem*}

\begin{proof}
Assume the hypotheses. The (final case of the) following proof is essentially a syntactic version of the proof of \cite[Lemma III.5.5]{MM}. We prove the desired result by induction on the structure of $t \in \Term^c\left(\Sigmasite(F, \x_C)\right)$ with $\Tsite(F, \x_C) \vdash t \downarrow$.
\begin{itemize}
\item If $t \equiv \x_C : C$, so that $\Tsite(F, \x_C) \vdash t \downarrow$, then we have the maximal sieve $t_C \in \J(C)$ with $\Tsite(F, \x_C) \vdash \alpha_f(\x_C) \downarrow$ for every $f \in t_C$, and for any $f \in t_C$ and $g \in \Arr(\C)$ with $\cod(g) = \dom(f)$ we have $\Tsite(F, \x_C) \vdash \alpha_g(\alpha_f(\x_C)) = \alpha_{f \circ g}(\x_C)$. It then follows that \[ \Tsite(F, \x_C) \vdash \sigma_{t_C}\left(\left(\alpha_f(\x_C)\right)_{f \in t_C}\right) \downarrow\] and \[ \Tsite(F, \x_C) \vdash \alpha_f\left(\sigma_{t_C}\left(\left(\alpha_f(\x_C)\right)_{f \in t_C}\right)\right) = \alpha_f(\x_C) \] for every $f \in t_C$, so that $\Tsite(F, \x_C) \vdash \x_C = \sigma_{t_C}\left(\left(\alpha_f(\x_C)\right)_{f \in t_C}\right)$ by the uniqueness of amalgamations, as desired.  

\item If $t \equiv c_a : D$ for some $D \in \Ob(\C)$ and $a \in F(D)$, then the reasoning is identical to the first case.

\item Suppose that $t \equiv \alpha_k(s) : E$ for some morphism $k : E \to D$ in $\C$ and some term $s \in \Term^c\left(\Sigmasite(F, \x_C)\right)$ with $s : D$, and suppose that $\Tsite(F, \x_C) \vdash t \downarrow$, which implies that $\Tsite(F, \x_C) \vdash s \downarrow$. By the induction hypothesis, there is some cover $J \in \J(D)$ with $\Tsite(F, \x_C) \vdash s = \sigma_J\left(\left(s_h\right)_{h \in J}\right)$, where for any $h \in J$ either $\Tsite(F, \x_C) \vdash s_h = c_a$ for some $a \in F(\dom(h))$ or $\Tsite(F, \x_C) \vdash s_h = \alpha_f(\x_C)$ for some morphism $f : \dom(h) \to C$ in $\C$.  

We then have the pullback sieve $k^*J \in \J(E)$ consisting of those morphisms $f$ with codomain $E$ for which $k \circ f \in J$, with $\Tsite(F, \x_C) \vdash \alpha_f(\alpha_k(s)) \downarrow$ for every $f \in k^*J$, and for any $f \in k^*J$ and $g \in \Arr(\C)$ with $\cod(g) = \dom(f)$ we have $\Tsite(F, \x_C) \vdash \alpha_g(\alpha_f(\alpha_k(s))) = \alpha_{f \circ g}(\alpha_k(s))$. It then follows that \[ \Tsite(F, \x_C) \vdash \sigma_{k^*J}\left(\left(\alpha_f(\alpha_k(s))\right)_{f \in k^*J}\right) \downarrow\] and \[ \Tsite(F, \x_C) \vdash \alpha_f\left(\sigma_{k^*J}\left(\left(\alpha_f(\alpha_k(s))\right)_{f \in k^*J}\right)\right) = \alpha_f(\alpha_k(s)) \] for every $f \in k^*J$, so that \[ \Tsite(F, \x_C) \vdash \alpha_k(s) = \sigma_{k^*J}\left(\left(\alpha_f(\alpha_k(s))\right)_{f \in k^*J}\right) = \sigma_{k^*J}\left(\left(\alpha_{k \circ f}(s)\right)_{f \in k^*J}\right)  \] by the uniqueness of amalgamations. But for any $f \in k^*J$ we have $k \circ f \in J$, so that $\Tsite(F, \x_C) \vdash \alpha_{k \circ f}(s) = s_{k \circ f}$, and hence by the induction hypothesis it follows that $t \equiv \alpha_k(s)$ is provably equal to a term of the desired form.   

\item Lastly, suppose that $t \equiv \sigma_J\left(\left(t_h\right)_{h \in J}\right) : D$ for some $D \in \Ob(\D)$ and some cover $J \in \J(D)$, with $t_h \in \Term^c\left(\Sigmasite(F, \x_C)\right)$ of sort $\dom(h)$ for every $h \in J$. Assume that $\Tsite(F, \x_C) \vdash t \downarrow$, so that $\Tsite(F, \x_C) \vdash t_h \downarrow$ and $\Tsite(F, \x_C) \vdash \alpha_h(t) = t_h$ for every $h \in J$, and $\Tsite(F, \x_C) \vdash \alpha_g(t_h) = t_{h \circ g}$ for every $h \in J$ and $g \in \Arr(\C)$ with $\cod(g) = \dom(h)$. 

By the induction hypothesis, we know for every $h \in J$ that there are some cover $J_h \in \J(\dom(h))$ and terms $t_{h, k} \in \Term^c\left(\Sigmasite(F, \x_C)\right)$ of sort $\dom(k)$ for every $k \in J_h$ with $\Tsite(F, \x_C) \vdash t_h = \sigma_{J_h}\left(\left(t_{h, k}\right)_{k \in J_h}\right)$, and each $t_{h, k}$ is provably equal to an object constant of $F$ or a term of the form $\alpha_\ell(\x_C)$ for some morphism $\ell : \dom(k) \to C$ of $\C$. 

By the reasoning in the previous case, for every $h \in J$ and $g \in \Arr(\C)$ with $\cod(g) = \dom(h)$ we have that $\Tsite(F, \x_C)$ proves the equations
\begin{align*}
&\ \ \ \ \sigma_{J_{h \circ g}}\left(\left(t_{h \circ g, k}\right)_{k \in J_{h \circ g}}\right) \\ 
&= t_{h \circ g} \\ 
&= \alpha_g(t_h) \\ 
&= \alpha_g\left(\sigma_{J_h}\left(\left(t_{h, k}\right)_{k \in J_h}\right)\right) \\ 
&= \sigma_{g^*J_h}\left(\left(t_{h, g \circ k}\right)_{k \in g^*J_h}\right).
\end{align*}
Considering the cover $T_{h, g} := J_{h \circ g} \cap g^*J_h \in \J(\dom(g))$, for any $k \in T_{h, g}$ we then have that $\Tsite(F, \x_C)$ proves
\[ t_{h \circ g, k} = \alpha_k\left(\sigma_{J_{h \circ g}}\left(\left(t_{h \circ g, k}\right)_{k \in J_{h \circ g}}\right)\right) = \alpha_k\left(\sigma_{g^*J_h}\left(\left(t_{h, g \circ k}\right)_{k \in g^*J_h}\right)\right) = t_{h, g \circ k}. \]

Now consider the sieve $K := \left\{h \circ k \colon h \in J, k \in J_h\right\}$ on $D$. Since $J \in \J(D)$ and $J_h \in \J(\dom(h))$ for all $h \in J$, it follows by the transitivity axiom for a Grothendieck topology that $K \in \J(D)$. For any $h \in J$ and $k \in J_h$, we have $\Tsite(F, \x_C) \vdash t_{h, k} \downarrow$. We now claim that if $h, h' \in J$ and $k \in J_h, k' \in J_{h'}$ with $h \circ k = h' \circ k'$, then $\Tsite(F, \x_C) \vdash t_{h, k} = t_{h', k'}$. Indeed, for any $g \in T_{h, k} \cap T_{h', k'}$ we have 
\[ \Tsite(F, \x_C) \vdash \alpha_g\left(t_{h, k}\right) = t_{h, k \circ g} = t_{h \circ k, g} = t_{h' \circ k', g} = t_{h', k' \circ g} = \alpha_g\left(t_{h', k'}\right). \] Since $T_{h, k} \cap T_{h', k'} \in \J(\dom(k)) = \J(\dom(k'))$, we then conclude from the fourth group of axioms for $\Tsite$ that $\Tsite(F, \x_C) \vdash t_{h, k} = t_{h', k'}$, as desired. 

For any $h \in J$ and $k \in J_h$, we now set $s_{h \circ k} := t_{h, k}$, which is well-defined up to provable equality in $\Tsite(F, \x_C)$ by what we just showed. For any $g \in \Arr(\C)$ with $\cod(g) = \dom(h \circ k)$ we also have 
\[ \Tsite(F, \x_C) \vdash \alpha_g(s_{h \circ k}) = \alpha_g(t_{h, k}) = t_{h, k \circ g} = s_{h \circ (k \circ g)} = s_{(h \circ k) \circ g}, \] which entails that $\Tsite(F, \x_C) \vdash \sigma_K\left(\left(s_{h \circ k}\right)_{h \circ k \in K}\right) \downarrow$. Given that each $s_{h \circ k} \equiv t_{h, k}$ has the desired form by the induction hypothesis, it remains to prove that $\Tsite(F, \x_C) \vdash \sigma_J\left(\left(t_h\right)_{h \in J}\right) = \sigma_K\left(\left(s_{h \circ k}\right)_{h \circ k \in K}\right)$. By the uniqueness of amalgamations, it suffices to show for any $h' \in J$ that $\Tsite(F, \x_C) \vdash t_{h'} = \alpha_{h'}\left(\sigma_K\left(\left(s_{h \circ k}\right)_{h \circ k \in K}\right)\right)$, i.e. that $\Tsite(F, \x_C) \vdash \sigma_{J_{h'}}\left(\left(t_{h', k}\right)_{k \in J_{h'}}\right) = \alpha_{h'}\left(\sigma_K\left(\left(s_{h \circ k}\right)_{h \circ k \in K}\right)\right)$. And to show \emph{this}, it again suffices by uniqueness of amalgamations to show for any $k \in J_{h'}$ that $\Tsite(F, \x_C) \vdash t_{h', k} = \alpha_k\left(\alpha_{h'}\left(\sigma_K\left(\left(s_{h \circ k}\right)_{h \circ k \in K}\right)\right)\right)$, which just follows because $h' \circ k \in K$ and $\Tsite(F, \x_C) \vdash s_{h' \circ k} = t_{h', k}$. This completes the proof.     
\end{itemize} 
\end{proof}

\begin{lem*}[\ref{invertiblelemma}]
Let $(\C, \J)$ be a small subcanonical site in which no object is covered by the empty sieve, and let $F \in \Sh(\C, \J)$ and $C \in \Ob(\C)$. For any closed term $t \in \Term^c\left(\Sigmasite(F, \x_C)\right)$ with $\Tsite(F, \x_C) \vdash t \downarrow$ and $t : C$, if there is some term $s \in \Term^c\left(\Sigmasite(F, \x_C)\right)$ with $\Tsite(F, \x_C) \vdash s \downarrow$ and $s : C$ and $\Tsite(F, \x_C) \vdash t[s/\x_C] = \x_C$, then there is a \textbf{pure} term $t' \in \Term^c\left(\Sigmasite(F, \x_C)\right)$ with $\Tsite(F, \x_C) \vdash t = t'$.  
\end{lem*}

\begin{proof}
Assume the hypotheses. By Lemma \ref{normalformlemma}, there is a cover $J \in \J(C)$ with $\Tsite(F, \x_C) \vdash t = \sigma_J\left(\left(t_h\right)_{h \in J}\right)$ for some terms $t_h \in \Term^c\left(\Sigmasite(F, \x_C)\right)$ of sort $\dom(h)$ for all $h \in J$ satisfying the conditions of Lemma \ref{normalformlemma}. Now fix $h \in J$. By assumption, we have
\[ \Tsite(F, \x_C) \vdash \alpha_h(\x_C) = \alpha_h\left(t[s/\x_C]\right) = \alpha_h\left(\sigma_J\left(\left(t_h[s/\x_C]\right)_{h \in J}\right)\right) = t_h[s/\x_C]. \] We now show that there can be no object $a \in F(\dom(h))$ with $\Tsite(F, \x_C) \vdash t_h = c_a$. For if there were, then we would have $\Tsite(F, \x_C) \vdash \alpha_h(\x_C) = t_h[s/\x_C] = c_a[s/\x_C] = c_a$, contradicting Lemma \ref{constantlemma}. Then by Lemma \ref{normalformlemma} we must have $\Tsite(F, \x_C) \vdash t_h = \alpha_{f_h}(\x_C)$ for some morphism $f_h : \dom(h) \to C$ in $\C$ (for every $h \in J$). We therefore have
\[ \Tsite(F, \x_C) \vdash t = \sigma_J\left(\left(t_h\right)_{h \in J}\right) = \sigma_J\left(\left(\alpha_{f_h}(\x_C)\right)_{h \in J}\right), \] with the latter term being pure.    
\end{proof}


\begin{thebibliography}{1}

\bibitem{LPAC}
Adamek, J., Rosicky, J.:
\emph{Locally Presentable and Accessible Categories}.
London Mathematical Society Lecture Note Series 189, Cambridge University Press (1994).

\bibitem{Berg}
van den Berg, B., Moerdijk, I.:
Aspects of predicative algebraic set theory III: sheaves.
Proc. Lond. Math. Soc. Vol. 105, Issue 5, 1076-1122 (2012).

\bibitem{Bergman}
Bergman, G.: 
An inner automorphism is only an inner automorphism, but an inner endomorphism can be
something strange. 
Publicacions Matematiques 56, 91-126 (2012).

\bibitem{Borceux}
Borceux, F.:
\emph{Handbook of Categorical Algebra 3: Categories of Sheaves}.
Encyclopedia of Mathematics and its Applications 52, Cambridge University Press, Cambridge (1994).

\bibitem{Bridge}
Bridge, P.:
Essentially algebraic theories and localizations in toposes and abelian categories.
PhD Thesis, University of Manchester (2012).

\bibitem{Funk}
Funk, J., Hofstra, P., Steinberg, B.: 
Isotropy and crossed toposes. 
Theor. Appl. Cat. 26, 660-709 (2012).

\bibitem{MFPSpaper}
Hofstra, P., Parker, J., Scott, P. J.:
Isotropy of algebraic theories. 
Electron. Notes Theor. Comput. Sci. 341, 201-217 (2018).

\bibitem{MM}
Mac Lane, S., Moerdijk, I.:
\emph{Sheaves in Geometry and Logic: A First Introduction to Topos Theory}.
Springer-Verlag New York Inc. (1992). 

\bibitem{Horn}
Palmgren, E., Vickers, S. J.: 
Partial Horn logic and cartesian categories.
Ann. Pure Appl. Log. 145, 314-353 (2007).

\bibitem{thesis}
Parker, J.: 
Isotropy groups of quasi-equational theories.
PhD Thesis, University of Ottawa (2020).

\end{thebibliography}
\end{document}